\documentclass[12pt,reqno,fleqn]{amsart}  
\usepackage{tikz}
\usepackage{amsmath,amstext,amsthm,amssymb,amsxtra}
\usepackage[top=1.5in, bottom=1.5in, left=1.25in, right=1.25in]	{geometry}
\usepackage{lmodern}

\usepackage{mathtools}
\mathtoolsset{showonlyrefs,showmanualtags}

\usepackage{hyperref} %,hypertexnames=false,colorlinks,[pagebackref]
\hypersetup{
	colorlinks=true,       % false: boxed links; true: colored links
	linkcolor=blue,          % color of internal links
	citecolor=magenta,        % color of links to bibliography
	filecolor=magenta,      % color of file links
	urlcolor=cyan           % color of external links
}

\usepackage[msc-links]{amsrefs}

\newtheorem{theorem}{Theorem}[section]
\newtheorem{proposition}{Proposition}[section]
\newtheorem{lemma}{Lemma}[section]
\newtheorem{corollary}{Corollary}[section]
\newtheorem{OldTheorem}{Theorem}
\theoremstyle{definition}
\newtheorem{definition}{Definition}[section]
\theoremstyle{definition}

\theoremstyle{remark}
\newtheorem{remark}{Remark}[section]

\def\spec{{\rm spec}}
\def\sign{{\rm sign\,}}
\def\pv{{\rm p.v. }}

\def\OSC{{\rm OSC}}

\def\supp{{\rm supp\,}}

\def\dist{{\rm dist}}

\def\ZZ{\ensuremath{\mathbb Z}}
\def\ZI{\ensuremath{\textbf 1}}
\def\x{\ensuremath{\textbf x}}
\def\u{\ensuremath{\textbf u}}
\def\t{\ensuremath{\textbf t}}
\def\h{\ensuremath{\textbf h}}
\def\n{\ensuremath{\textbf n}}
\def\ZN{\ensuremath{\mathbb N}}

\def\ZP{\ensuremath{\mathcal P}}

\def\ZK{\ensuremath{\mathcal K}}
\def\ZR{\ensuremath{\mathbb R}}

\def\ZT{\ensuremath{\mathbb T}}

\def\Re{\ensuremath{\mathrm {Re }}}
\def\Im{\ensuremath{\mathrm {Im }}}

\numberwithin{equation}{section}
\def\md#1#2\emd{\ifx0#1
	\begin{equation*} #2 \end{equation*}\fi  %  single line display, no number
	\ifx1#1\begin{equation}#2\end{equation}\fi   % single line display, number
	\ifx2#1\begin{align*}#2\end{align*}\fi   % aligned display, no number
	\ifx3#1\begin{align}#2\end{align}\fi    % aligned display, number
	\ifx4#1\begin{gather*}#2\end{gather*}\fi  % multline, not align, no number
	\ifx5#1\begin{gather}#2\end{gather}\fi   % multinline, not align
	\ifx6#1\begin{multline*}#2\end{multline*}\fi  %  display too long for one line
	\ifx7#1\begin{multline}#2\end{multline}\fi  % as above, with numbers
	\ifx8#1\begin{multline*}\begin{split}#2\end{split}\end{multline*}\fi
	\ifx9#1\begin{multline}\begin{split}#2\end{split}\end{multline}\fi
}
\newcommand {\e }[1]{\eqref{#1}}
\newcommand {\lem }[1]{Lemma \ref{#1}}

\newcommand {\cor }[1]{Corollary \ref{#1}}
\newcommand {\pro }[1]{Proposition \ref{#1}}
\newcommand {\trm }[1]{Theorem \ref{#1}}

\title[] {On Weyl multipliers of the rearranged trigonometric system}
\author{Grigori A. Karagulyan}
\address{Faculty of Mathematics and Mechanics, Yerevan State
	University, Alex Manoogian, 1, 0025, Yerevan, Armenia} 
\email{g.karagulyan@ysu.am}
\address{Institute of Mathematics of NAS of RA, Marshal Baghramian ave., 24/5, Yerevan, 0019, Armenia} 
\email{g.karagulyan@gmail.com}

\subjclass[2010]{42C05, 42C10, 42C20}
\keywords{Trigonometric series, Weyl multipliers, Menshov-Rademacher theorem}
\begin{document}
\begin{abstract}
We prove that the condition
\begin{equation}
\sum_{n=1}^\infty\frac{1}{nw(n)}<\infty
\end{equation}
is necessary for an increasing sequence of numbers $w(n)$ to be an almost everywhere unconditional convergence Weyl multiplier for the trigonometric system. This property for Haar, Walsh,  Franklin and some other classical orthogonal systems was known long ago. The proof of this result is based on a new sharp logarithmic lower bound on $L^2$ of the majorant operator related to the rearranged trigonometric system.
\end{abstract}

	\maketitle  
%%%%%%%%%%%%%%%%%%%%%%%%%%%%%% SECTION  SECTION SECTION
%%%%%%%%%%%%%%%%%%%%%%%%%%%%%% SECTION  SECTION SECTION 
\section{Introduction}
	Let $\Phi=\{\phi_n:\, n=1,2,\ldots\}\subset L^2(0,1)$ be an orthonormal system. Recall that a sequence of positive numbers $w(n)\nearrow\infty$ is said to be an a.e. convergence Weyl multiplier (shortly C-multiplier) if every series 
	\begin{equation}\label{a1}
	\sum_{n=1}^\infty a_n\phi_n(x),
	\end{equation}
	with coefficients satisfying the condition 
	\begin{equation}\label{a3}
	\sum_{n=1}^\infty a_n^2w(n)<\infty
	\end{equation}
	is a.e. convergent (see \cite{KaSa} or \cite{KaSt}). 
	The Menshov-Rademacher classical theorem (\cite{Men},  \cite{Rad}) states that the sequence $\log^2 n$ is a C-multiplier for any orthonormal system. The sharpness of $\log^2 n$ in this statement was established by Menshov in the same paper \cite{Men}, proving that any sequence $w(n)=o(\log^2n)$ fails to be C-multiplier for some  orthonormal system. The following definitions are well known in the theory of orthogonal series. 
	\begin{definition}
	A sequence of positive numbers $w(n)\nearrow\infty$ is said to be an a.e. convergence Weyl multiplier for the rearrangements (RC-multiplier) of an orthonormal system $\Phi$ if it is C-multiplier for any rearrangement of $\Phi$.
	\end{definition}
\begin{definition}
	A sequence of positive numbers $w(n)\nearrow\infty$ is said to be an a.e. unconditional convergence Weyl multiplier (UC-multiplier) for an orthonormal system $\Phi$ if  under the condition \e{a3} the series \e{a1} converge a.e. after any rearrangement of its terms.
\end{definition}
For a given orthonormal system $\Phi$, we denote by ${\rm RC}(\Phi)$ and ${\rm UC}(\Phi)$ the families of RC and UC multipliers, respectively. Observe that according to the Menshov-Rademacher theorem we have $\log^2 n\in {\rm RC}(\Phi)$ for any orthonormal system $\Phi$ and the counterexample of Menshov tells us that $\log^2 n$ is optimal in this statement. The following two theorems provide a necessary and sufficient condition for a sequence to be UC-multiplier for all orthonormal systems. Namely,
\begin{OldTheorem}[Orlicz, \cite{Orl}]
If an increasing sequence of positive numbers $\lambda(n)$ satisfies
	\begin{equation}\label{d3}
\sum_{n=1}^\infty\frac{1}{n\lambda(n)\log n}<\infty,
\end{equation} 
then $w(n)=\lambda(n)\log^2n$ is a UC-multiplier for any orthonormal system.
\end{OldTheorem}
\begin{OldTheorem}[Tandori, \cite{Tan}]
	If an increasing sequence of positive numbers $\lambda(n)$ doesn't satisfy \e{d3},  then there exists an orthonormal system for which the sequence $w(n)=\lambda(n)\log^2n$ fails to be a UC-multiplier.
\end{OldTheorem}
In particular, these results imply that the sequence $\log^2n(\log\log n)^{1+\varepsilon}$, $\varepsilon >0$ is an UC-multiplier for any orthonormal system, while $\log^2n\log\log n$ is not a UC-multiplier for some orthonormal systems.

The study of RC and UC multipliers of classical orthonormal systems is an old issue in the theory of orthogonal series. It is well known that the sequence $w(n)\equiv 1$ is a C-multiplier for trigonometric, Walsh, Haar and Franklin systems, while it fails to be RC-multiplier for those systems. Kolmogorov \cite{Kol} was the first who remarked that the sequence $w(n)\equiv 1$ is not RC-multiplier for the trigonometric system. However, he has never published the proof of this fact. A proof of this assertion was later given by Zahorski \cite{Zag}. Afterward, developing Zahorski's argument, Ul\cprime yanov \cite{Uly6, Uly7} established such a property for Haar and Walsh systems. Using the Haar functions technique, Olevskii \cite{Ole} succeed proving that such a phenomenon is common for arbitrary complete orthonormal system.

Later on Ul\cprime yanov \cite{Uly1, Uly4} found the optimal growth of the RC and UC multipliers of Haar system. Moreover, his  technique of the proof became a key argument in the study of the analogous problems for other classical systems.
Namely, 
\begin{OldTheorem}[Ul\cprime yanov, \cite{Uly3}]\label{OT1}
	The sequence $\log n$ is an RC-multiplier for the Haar system and any sequence $w(n)=o(\log n)$ is not.
\end{OldTheorem}
\begin{OldTheorem}[Ul\cprime yanov, \cite{Uly3}]\label{OT2}
	The sequence $w(n)$ is a UC-multiplier for the Haar system if and only if it holds the bound
	\begin{equation}\label{a4}
	\sum_{n=1}^\infty\frac{1}{nw(n)}<\infty.
	\end{equation}
\end{OldTheorem}

  In his famous overview \cite {Uly5} of 1964 Ul\cprime yanov raised two problems (see \cite{Uly5}, pp. 58, 62-63), and those have been further recalled several times in different papers of the author (see \cite{Uly4}, p. 1041, \cite{Uly8}, p. 80, \cite{Uly9}, p. 57). The problems claim

 1) Find the optimal sequence $w(n)$ to be RC-multiplier for the trigonometric (Walsh) system?

 2) Characterize the UC-multipliers of the trigonometric (Walsh) system?
 
 The following result somehow clarifies the relationship between these two problems in the terms of the Orlicz "extra factor" $\lambda(n)$ (see \e{d3}). It also tells us that the Orlicz theorem can be deduced from the Menshov-Rademacher theorem. 
 \begin{OldTheorem}[Ul\cprime yanov-Poleshchuk, \cite{Uly3, Pol}]\label{OT3}
 	If $w(n)$ is an RC-multiplier for an orthonormal system $\Phi=\{\phi_n(x)\}$ and $\lambda(n)$ is an increasing sequence of positive numbers  satisfying \e{d3},
 	then the sequence $\lambda(n)w(n)$ is a UC-multiplier for $\Phi$.
 \end{OldTheorem}
Relating to the problem 1), we first note that the Menshov-Rademacher theorem implies that $\log^2 n$ is a RC-multiplier for the trigonometric and Walsh systems, and second, it is not known any RC-multiplier $w(n)=o(\log^2n)$ for those systems. Similarly, the only known UC-multipliers of trigonometric and Walsh systems are sequences $\lambda(n)\log^2 n$ coming form the result of Orlicz for the general orthonormal systems. 

The lower estimates for RC and UC multipliers of Walsh system were studied in \cite{Boch, Nak4, Nak5, Nak3, Tan3}. The best result at this moment proved independently by Bochkarev \cite{Boch2, Boch} and Nakata \cite{Nak3} says that if an increasing sequence $w(n)$ satisfies 
\begin{equation}\label{4}
\sum_{n=1}^\infty\frac{1}{nw(n)}=\infty,
\end{equation}
then it is not an UC-multiplier for the Walsh system.

For the trigonometric system analogous bounds were studied in \cite{Mor, Tan, Nak1, Nak2, Seroj}. The most general result is due Galstyan \cite{Seroj}, 1992, who proved that under the condition 
\begin{equation}\label{y3}
\sum_{n=1}^\infty\frac{1}{n\log\log n\cdot w(n)}=\infty
\end{equation}
the sequence $w(n)$ fails to be UC-multiplier for the trigonometric system. In contrast to Haar and Walsh systems in the trigonometric case we see here extra $\log\log n$ factor in \e{y3}. The \cor{C2} stated below tells us that the factor $\log\log n$ can be removed also in the trigonometric system case.

Note that the following inequality is the key argument in the proof of the Menshov-Rademacher theorem.
\begin{OldTheorem}[Menshov-Rademacher, \cite{Men}, \cite{Rad}, see also \cite{KaSa}]\label{MR}
	For any orthonormal system $\{\phi_k:\, \,k=1,2,\ldots,n\}\subset L^2(0,1)$ and coefficients $a_k$ it holds
	\begin{equation}\label{1}
	\left\|\max_{1\le m \le n}\left|\sum_{k=1}^ma_k\phi_k\right|\,\right\|_2\le c\cdot\log n \left\|\sum_{k=1}^na_k\phi_k\right\|_2,
	\end{equation}
	where $c>0$ is an absolute constant.
\end{OldTheorem}
Similarly, the counterexample of Menshov is based on the following results.
\begin{OldTheorem}[Menshov, \cite{Men}]\label{M}
	For any natural number $n\in\ZN$ there exists an orthogonal system $\phi_k$, $k=1,2,\ldots,n$, such that
	\begin{equation}
	\left\|\max_{1\le m \le n}\left|\sum_{k=1}^m\phi_k\right|\,\right\|_2\ge c\cdot\log n \left\|\sum_{k=1}^n\phi_k\right\|_2,
	\end{equation}
	for an absolute constant $c>0$.
\end{OldTheorem}

To state the results of the present paper let us introduce some notations. For two positive quantities $a$ and $b$ the notation $a\lesssim b$ will stand for the inequality $a<c\cdot b$, where $c>0$ is an absolute constant, and we write $a\sim b$ whenever $a\lesssim b\lesssim a$. 
Let $\Sigma_N$ denote the family of one to one mappings (permutations) on $\{1,2,\ldots,N\}$. We will consider the trigonometric system on the torus $\ZT=\ZR/\ZZ$. For a given integer $N\ge 1$ and $\sigma\in\Sigma_N$ we consider the operator $T_{\sigma,N}:L^2(\ZT)\to L^2(\ZT)$ defined by 
\begin{align}
T_{\sigma,N}f(x)=\max_{1\le m\le N}\left|\sum_{k=1}^mc_{\sigma(k)}e^{2\pi i\sigma(k)x}\right|\text { where }c_k=\int_\ZT f(x)e^{-2\pi ikx}dx.
\end{align}

Our main result is the following.
\begin{theorem}\label{T3}
	For any integer $N> 1$ there exists a permutation $\sigma\in \Sigma_N$ such that
	\begin{align}
	\|T_{\sigma,N}\|_{L^2\to L^2}\sim  \log N.\label{a50}
	\end{align}
\end{theorem}
We note that the upper bound in \e{a50} follows from the Menshov-Rademacher inequality \e{1}. Recall the weak $L^2$-norm of an operator $T:L^2\to L^2$ defined
\begin{equation*}
\|T\|_{L^2\to L^{2,\infty}}=\sup_{\|f\|_2\le 1,\,\lambda>0}\lambda(|\{|Tf(x)|>\lambda\}|)^{1/2}.
\end{equation*}
From the lower bound of \e{a50}, applying  \lem{L11}, we easily deduce also a lower estimate for the weak $L^2$-norm of the operator $T_{\sigma,N}$. Namely,
\begin{corollary}\label{C0}
For any integer $N> 1$ one can find a permutation  $\sigma\in \Sigma_N$ such that
\begin{equation}\label{a51}
\|T_{\sigma,N}\|_{L^2\to L^{2,\infty}}\gtrsim \sqrt{\log N}
\end{equation}
\end{corollary}
Applying \e{a51} we prove the following results.
\begin{corollary}\label{C1}
 Any increasing sequence of positive numbers $w(n)$, satisfying  
 \begin{equation}\label{y1}
w(n)=o(\log n),
 \end{equation}
 fails to be a RC-multiplier for the trigonometric system. Moreover, there are coefficients $a_n$, satisfying \e{a3}, such that the series
 \begin{equation*}
 \sum_{n=1}^\infty a_ne^{2\pi i\sigma(n)x}
 \end{equation*}
 is almost everywhere divergent for some permutation $\sigma$.
\end{corollary}
\begin{corollary}\label{C2}
	If  an increasing sequence of positive numbers $w(n)$ satisfies \e{4},
then it is not UC-multiplier for the trigonometric system. Namely, there are coefficients $a_n$ satisfying \e{a3} such that the series  
\begin{equation}\label{u25}
\sum_{n=1}^\infty a_ne^{2\pi inx}
\end{equation}
can be rearranged into almost everywhere divergent series.
\end{corollary}
\begin{remark}
	Corollaries \ref{C1}, \ref{C2} can be stated in the terms of real trigonometric series, considering
	\begin{equation*}
	\sum_{n=1}^\infty a_n\cos(nx +\rho_n)
	\end{equation*}
instead of series \e{u25}. In fact, using an elementary argument, one can deduce the real trigonometric versions of Corollaries \ref{C1} and \ref{C2} from their complex analogs.
\end{remark}
\begin{remark}
	We do not know wheather it holds the converse inequality of \e{a51} for any permutation $\sigma$, i.e.
	\begin{equation*}
	\max_{\sigma\in \Sigma_N }\|T_{\sigma,N}\|_{L^2\to L^{2,\infty}}\lesssim \sqrt{\log N}.
	\end{equation*}
\end{remark}
\begin{remark}
It is also not known the estimate like \e{a51} for the Walsh system. Note that our proof of \e{a51} is based on a specific argument that is common only for the trigonometric system and it is not applicable in the case of Walsh system. Namely, we use a logarithmic lower bound for the directional Hilbert transform on the plane due to Demeter \cite{Dem}. 
\end{remark}
\begin{remark}
	Recall the following problem posed Kashin in \cite{Kas}, that became more interesting after the result of \trm{T3}: is there a sequence of positive numbers $\gamma(n)=o(\log n)$ such that for any orthonormal system $\phi_n$ on $(0,1)$ it holds the inequality
	\begin{equation}
	\left(\int_0^1\int_0^1\max_{1\le m \le n}\left|\sum_{k=1}^m\phi_k(x)\phi_k(y)\right|^2dxdy\right)^{1/2}\le \gamma(n)\sqrt n?
	\end{equation}

\end{remark}
\begin{remark}
	Finally, we note that an analogous result of \trm{OT2} for the Franklin system was proved by Gevorkyan \cite{Gev}. In a recent paper of author \cite{Kar} it was proved the analogous of the Theorems \ref{OT1} and \ref{OT2} for the orthonormal systems of non-overlapping martingale-difference (in particular, Haar) polynomials.
\end{remark}

I would like to thank Boris Kashin for the discussions on the subject, as well as the referees for careful reading and valuable remarks improving the quality of the paper.

\section{Directional Hilbert transform and Demeter's example}
The starting point for our construction is an example given by Demeter \cite{Dem} for the directional Hilbert transform. 
To state it we need the notations 
\begin{align}
&B(a,b)=\{\textbf{x}\in \ZR^2:\, a\le  \|\textbf{x}\| <b\},\,0\le a<b\le \infty,\\
&\Gamma_\theta=\{\textbf{x}=(x_1,x_2)\in \ZR^2:\,x_1\cos \theta +x_2\sin\theta \ge 0\}.
\end{align}
For a rapidly decreasing function $f$ and a unit vector $(\cos\theta , \sin \theta  ),\,\theta\in [0,2\pi),$ we define
\begin{equation}
H_\theta f(\textbf{x})=\pv\frac{1}{\pi}\int_{\ZR } \frac{f(\textbf{x}-t(\cos \theta,\sin  \theta ))}{t}dt,\quad \textbf{x}=(x_1,x_2)\in \ZR^2,
\end{equation}
which is the one dimensional Hilbert transform corresponding to the direction
$\theta$. It is well known this operator can be extended to a bounded
operator on $L^2(\ZR^2)$. For the collection of uniformly distributed unit vectors
\begin{equation}
\Theta=\{\theta_k=\pi k/N,\, k=1,2,\ldots, N\},
\end{equation}
consider the operator
\begin{equation*}\label{HU}
H_\Theta^*f(\x)=\sup_{\theta\in \Theta}|H_\theta f(\x)|.
\end{equation*}
The result of the paper \cite{Dem} is the lower bound $\|H_\Theta^*\|_{2\to 2}\gtrsim \log N$. We find suitable to give a detailed proof of this result. 
\begin{lemma}[Demeter, \cite{Dem}]
	For any integer $N>N_0$, where $N_0$ is an absolute constant,  the function 
	\begin{equation}\label{x8}
	f(\x)=\frac{1}{\|\x\|}\cdot \ZI_{B(10N^{-9},N^{-8})}(\x)
	\end{equation}
	satisfies the inequality
	\begin{equation}\label{x9}
	\|H_\Theta^*(f)\|_2\gtrsim \log N\cdot \|f\|_2.
	\end{equation}
\end{lemma}
\begin{proof}
	A change of variable enables to prove \e{x9} for the function
	\begin{equation*}
	h(\x)=\frac{1}{\|\x\|}\cdot \ZI_{B(10,N)}(\x)
	\end{equation*}
instead of $f$.	Fix a point $\x$ in the upper half-plane  $\ZR^2_+$ satisfying 
\begin{align}
10^5\le \|\x\|\le N/3,\label{u5}
\end{align}
and consider the unit vector $\textbf{u}=\frac{\x}{\|\x\|}=(\cos \theta,\sin  \theta )$. Clearly, there is a unit vector $\textbf{u}_k=(\cos\theta_k,\sin\theta_k) $ such that 
\begin{equation}\label{y4}
\|\textbf{u}_k-\textbf{u}\|\le \pi/N.
\end{equation}
A geometric argument shows that the line $\x-t\textbf{u}_k$, $t\in \ZR$, has four intersection points with the boundary of $B(10,N)$. Moreover, we have
\begin{equation}\label{x10}
E=\{t\in \ZR:\, \x-t\textbf{u}_k\in B(10,N)\}=\big(A,a\big]\cup \big[b,B\big),
\end{equation}
where the numbers $A<a<b<B$ satisfy
\begin{align}
&\|\x-A\textbf{u}_k\|=\|\x-B\textbf{u}_k\|=N,\quad \|\x-a\textbf{u}_k\|=\|\x-b\textbf{u}_k\|=10,\label{y11}\\
&t\in (a,b) \Leftrightarrow \|\x-t\textbf{u}_k\|<10,\label{u8}\\
&t\not \in [A,B] \Leftrightarrow \|\x-t\textbf{u}_k\|>N.\label{u9}
\end{align}
For any $t\in E$ we have
\begin{equation}\label{y10}
|t|\le \|\x\|+\|\x-t\textbf{u}_k\|\le N/3+N=4N/3.
\end{equation}
Based on \e{y4}, \e{y11} and \e{y10}, we claim that
\begin{align}
&\|\x\|-15\le a\le \|\x\|-5,\label{u1}\\
&\|\x\|+5\le b\le \|\x\|+15,\label{u2}\\
&||A|-N|\le N/3,\label{u3}\\
&||B|-N|\le N/3.\label{u4}
\end{align}
Indeed, first observe that $t=\|\x\|\not\in E$, since
\begin{align}
\big\|\x-\|\x\|\textbf{u}_k\big\|\le \big\|\x-\|\x\|\textbf{u}\big\|+\|\x\|\|\textbf u-\textbf{u}_k\|=\|\x\|\|\textbf u-\textbf{u}_k\|\le 2.
\end{align}
So by \e{u8} we conclude $a<\|\x\|<b$. Thus, using also
\begin{align}
\big|10-|\|\x\|-a|\big|=\big|\|\x-a\textbf{u}_k\|-\|\x-a\textbf{u}\|\big|\le |a|\cdot \|\textbf u-\textbf{u}_k\|\le 5,
\end{align}
we easily get \e{u1}. Similarly we will have \e{u2}. From
\begin{align}
\big||A|-N\big|= \big|\|A\textbf{u}_k\|-\|\x-A\textbf{u}_k\|\big|\le \|\x\|\le\frac{N}{3},
\end{align}
and the same bound for $B$ we obtain \e{u3} and \e{u4} respectively. If $t\in E$, then by \e{y4} and \e{y10} we have 
\begin{equation*}
\big|\|\x\|-t\big|=\|\x-t\textbf{u}\|\ge \|\x-t\textbf{u}_k\|-|t|\|\textbf u-\textbf{u}_k\|\ge  10-5=5,
\end{equation*}
and therefore,
\begin{equation}
\|\textbf{x}-t\textbf{u}_k\|\ge \|\textbf{x}-t\textbf{u}\|-|t|\|\textbf{u}-\textbf{u}_k\|\ge \big|\|\x\|-t\big|-4,8\ge \frac{\big|\|\x\|-t\big|}{25}.
\end{equation}
Thus we get
\begin{equation*}
\left|\frac{1}{\|\textbf{x}-t\textbf{u}\|}-\frac{1}{\|\textbf{x}-t\textbf{u}_k\|}\right|\le \frac{|t|\|\textbf{u}-\textbf{u}_k\|}{\|\textbf{x}-t\textbf{u}\|\|\textbf{x}-t\textbf{u}_k\|}\le \frac{25\pi|t|}{N\big|t-\|\x\|\big|^2}
\end{equation*}
and hence, using also \e{x10}, \e{u1} and \e{u2},
\begin{align}
&\left|\pi\cdot H_{\theta_k} h(\textbf{x})-\pv\int_{E } \frac{1}{t\cdot \|\textbf{x}-t\textbf{u}\|}dt\right|\label{u6}\\
&\qquad\qquad\le \frac{25\pi}{N} \int_{E } \frac{1}{|t-\|\textbf{x}\||^2}dt\le\frac{50\pi}{N} \int_{5 }^\infty \frac{1}{t^2}dt= \frac{10\pi}{N}.
\end{align}
On the other hand
	\begin{align}
	\pv\int_{E } &\frac{1}{t\cdot \|\textbf{x}-t\textbf{u}\|}dt\\
	&=\pv \int_{A}^{a}\frac{dt}{t\cdot \|\textbf{x}-t\textbf{u}\|}+\pv \int_{b}^{B}\frac{dt}{t\cdot \|\textbf{x}-t\textbf{u}\|}\\
	&=\pv \int_{A}^{a}\frac{dt}{t\cdot (\|\textbf{x}\|-t)}+\pv \int_{b}^{B}\frac{dt}{t\cdot (t-\|\textbf{x}\|)}\\
	&=\pv \frac{1}{\|\x\|}\int_{A}^{a}\left(\frac{1}{\|\textbf{x}\|-t}+\frac{1}{t}\right)dt\\
	&\qquad+\frac{1}{\|\x\|} \int_{b}^{B}\left(\frac{1}{t-\|\textbf{x}\|}-\frac{1}{t}\right)dt\\
	&=\frac{1}{\|\x\|}\big(\log \big|\|\x\|-A\big|-\log \big|\|\x\|-a\big|+\log |a|-\log |A|\big)\\
	&\qquad +\frac{1}{\|\x\|}\big(\log \big|B-\|\x\|\big|-\log \big|b-\|\x\|\big|+\log |b|-\log |B|\big).
	\end{align}
Using \e{u5}, \e{u3} and \e{u4} we can say that 
	\begin{equation*}
	\log \big|\|\x\|-A\big|, \quad \log \big|B-\|\x\|\big|,\quad \log |A|,\quad \log |B|
	\end{equation*}
 are equal $\log N+c$ for different constants $c\in [\log(1/3),\log(5/3)]$. From \e{u1} and \e{u2} we get $\log \big|\|\x\|-a\big|,\,\log \big|b-\|\x\|\big|\in [\log 5,\log 15]$. While for $\log |a|$ and $\log |b|$ we have a lower bound by $\log( \|\x\|/2)$  in view of \e{u5}. All these imply
 \begin{equation}\label{u7}
 \pv\int_{E } \frac{1}{t\cdot \|\textbf{x}-t\textbf{u}\|}dt\ge\frac{2\log (10^{-5}\|\x\|)}{\|\x\|}.
 \end{equation}
Combining \e{u6} and \e{u7}, we obtain
	\begin{align}
	\pi\cdot H^*_{\Theta}h(\textbf{x})\ge \pi\cdot\left|H_{\theta_k} h(\textbf{x})\right|\ge \frac{\log  (10^{-5}\|\x\|)}{\|\x\|}-\frac{5\pi}{N}.
	\end{align}
	for all $\x\in\ZR^2_+$ satisfying \e{u5}. Thus, a simple integration shows that
	\begin{align*}
	\|H^*_{\Theta}(h)\|_2^2&\gtrsim \int_{B(10^{5},N/3)\cap \ZR^2_+}|H_{\theta_k}(h)|^2\\
&\gtrsim \int_{10^{5}}^{N/3}\left(\frac{\log^2 (10^{-5}r)}{r}- \frac{10\pi\log (10^{-5}r)}{N}+\frac{25\pi^2}{N^2}r\right)dr\\
&\gtrsim \log^{3}N
	\end{align*}
	and $\|h\|_2\lesssim \sqrt{\log N}$ for $N>N_0$. This implies \e{x9}.
\end{proof}

\section{Smooth modification of the function $f$ }
Since the one dimensional Hilbert transform is the multiplier operator of $i\cdot\sign x$,
for any direction $\theta=(\cos  \theta, \sin \theta )$ we have
\begin{equation}
\widehat{H_\theta f}(\x)=i\cdot \sign (x_1\cos \theta  +x_2\sin  \theta
)\widehat{f}(\x).
\end{equation}
Recall the multiplier operator $T_D$ corresponding to a region $D\subset \ZR^2$ defined by
\begin{equation*}
\widehat {T_D(f)}=\ZI_D\cdot \hat f.
\end{equation*}
One can check that
\begin{equation}\label{2-27}
T_{\Gamma_\theta}(f)=\frac{f-iH_\theta f}{2}.
\end{equation}
Let us denote
\begin{equation}
T^*f=\sup_{\theta \in \Theta }|T_{\Gamma_\theta} f|.
\end{equation}
So the bound \e{x9} is equivalent to the inequality
\begin{equation}\label{x5}
\|T^*(f)\|_{2}\gtrsim \log N\cdot \|f\|_2,
\end{equation}
which will be used in the next sections.
In this section we examine some properties of function \e{x8}.
\begin{lemma}
The function \e{x8} satisfies the relations
	\begin{align}
&\|f\|_1\sim N^{-8},\quad \|f\|_2\sim \sqrt{\log N},\label{x45}\\
&\omega_2(\delta,f)=\sup_{\|\h\|<\delta}\left(\int_{\ZR^2}|f(\x+\h)-f(\x)|^2d\x\right)^{1/2}\lesssim N^5\sqrt{\delta}.\label{x14}
\end{align}
for any $0<\delta<N^{-10}$.
\end{lemma}
\begin{proof}
	Equations \e{x45} are results of a simple integration. Fix a vector $\h$, $\|\h\|<\delta$. Observe that 
	\begin{equation}\label{x47}
	\x\in B(10N^{-9}+\delta, N^{-8}-\delta)
	\end{equation} 
	implies $\x+\h\in B(10N^{-9},N^{-8})$ and so
	\begin{equation*}
	|f(\x+\h)-f(\x)|=\left|\frac{1}{\|\x+\h\|}-\frac{1}{\|\x\|}\right|\le \frac{\|\h\|}{\|\x+\h\|\cdot \|\x\|}\le N^{18}\delta.
	\end{equation*}
	Using this we get
	\begin{align}\label{x49}
	\int_{B(10N^{-9}+\delta, N^{-8}-\delta)}&|f(\x+\h)-f(\x)|^2d\x\\
&\qquad \qquad 	\lesssim |B(10N^{-9}+\delta, N^{-8}-\delta)|\cdot N^{36}\delta^2\lesssim N^{20}\delta^2.
	\end{align}
If 
\begin{equation}\label{x48}
\x\in B(10N^{-9}-\delta, 10N^{-9}+\delta)\cup B(N^{-8}-\delta, N^{-8}+\delta),
\end{equation}
 then $|f(\x+\h)-f(\x)|\le 2\|f\|_\infty\lesssim N^{9}$ and so 
\begin{equation}\label{x50}
\int_{B(10N^{-9}+\delta, 10N^{-9}-\delta)\cup B(N^{-8}-\delta, N^{-8}+\delta)}|f(\x+\h)-f(\x)|^2dx\lesssim N^{10}\delta.
\end{equation}
If $\x$ is outside of the regions that we have in \e{x47} and \e{x48}, then $f(\x+\h)=f(\x)=0$. So combining \e{x49} and \e{x50}, we obtain \e{x14}.
\end{proof}
 It is well known there exists a spherical function $K\in L^\infty(\ZR^2)$ satisfying the relations
\begin{align}
&\int_{\ZR^2}K(\textbf{t})d\textbf{t}=1,\label{x28}\\
&\supp\hat K\subset B(0,1).\label{x19}\\
&0<K(\x)\le \frac{c}{|\x|^{50}},\label{x30}
\end{align} 
where $c>0$ is a constant. Indeed, chose a spherical function $\phi\in C^\infty (\ZR^2)$ with $\supp \phi \subset B(0,1/2)$ and define $K(x)$ by $\hat K=c_1(\phi\ast \phi)$. Clearly we will have \e{x19}, as well as \e{x30} for any power instead of $50$. The relation \e{x28} will be satisfied after a suitable choice of the constant $c_1$. 
We are going to replace the  function \e{x8} by the function 
\begin{equation}\label{x18}
g(\x)=\int_{\ZR^2}f(\x-\textbf{t})\ZK(\textbf{t})d\textbf{t}.
\end{equation}
where
\begin{equation*}
\ZK(\x)=N^{30}K\left(N^{15}\x\right).
\end{equation*}

\begin{lemma}\label{L1}
For large enough $N$ the function \e{x18} is spherical and satisfies the relations 
\begin{align}
&\supp \hat g\subset B(0,N^{15}),\label{x6}\\
&	\|g-f\|_2\lesssim\frac{1}{N^2},\label{x17}\\
&\|T^*(g)\|_{2}\gtrsim \log N\cdot \|g\|_2.\label{x7}
\end{align}
\end{lemma}
\begin{proof}
The function $g$ is spherical, since $f$ and $\ZK$ are spherical. Applying Fourier transform to the convolution \e{x18}, we get
	\begin{align}
	\hat g(\x )=\hat f(\x )\cdot \hat \ZK(\x )=\hat f(\x )\cdot \hat K\left(\frac{\x }{N^{15}}\right)\label{x31}
	\end{align}
so \e{x19} immediately implies \e{x6}. Write $g$ in the form
\begin{align}
g(\x )&=\int_{\ZR^2}f(\x -\textbf{t})\ZK(\textbf{t})d\textbf{t}\cdot \ZI_{B(0, 1)}(\x )\\
&\qquad+\int_{\ZR^2}f(\x -\textbf{t})\ZK(\textbf{t})d\t\cdot \ZI_{B(1,\infty )}(\x )\\
&=I_1(\x )+I_2(\x ).
\end{align}
Applying \e{x45} and \e{x30}, we can roughly estimate
\begin{equation}
|I_2(\x )|\lesssim \frac{ \ZI_{B(1,\infty )}(\x )}{N^{2}\|\x \|},
\end{equation}
then after a simple integration we get 
\begin{equation}\label{x55}
\|I_2\|_2\lesssim \frac{1}{N^2}.
\end{equation}
Choosing $\delta=N^{-14}$,  for every $\x \in B(0,1)$ we can write
	\begin{align}
	|I_1(\x )-f(\x )|&\le \int_{\ZR^2}|f(\x -\textbf{\textbf{t}})-f(\x )|\ZK(\textbf{t})d\textbf{t}\\
	&=\int_{B(0,\delta)}|f(\x -\textbf{t})-f(\x )|\ZK(\textbf{t})d\textbf{t}\\
	&\qquad +\int_{B(\delta,\infty)}|f(\x -\textbf{t})-f(\x )|\ZK(\textbf{t})d\textbf{t}\\
	&=I_{11}(\x )+I_{12}(\x ).
	\end{align}
From \e{x14} and \e{x28} it follows that
	\begin{align}
	\|I_{11}\|_2&\le \left(\int_{B(0,\delta)}\ZK(\textbf{t})\int_{\ZR^2}|f(\x -\textbf{\textbf{t}})-f(\x )|^2d\x d\textbf{\textbf{t}}\right)^{1/2}\\
	&\le \omega_2(\delta,f)\lesssim N^5\sqrt{\delta}\le \frac{1}{N^2}.\label{x15}
	\end{align}
Applying \e{x30} and the bound $\|f\|_\infty <N^{9}$, the second integral can be again roughly estimated as follows
	\begin{align}
|I_{12}(\x )|&\le 2N^{30}\|f\|_\infty \int_{B(\delta,\infty)}K(N^{15}\textbf{t})d\t\\
&\lesssim N^{39}\int_{B(\delta,\infty)}\frac{1}{(N^{15}|\t|)^{50}}dt\lesssim \frac{1}{N^{2}}
	\end{align}
and so
\begin{equation}\label{x16}
\|I_{12}\cdot \ZI_{B(0,1)}\|_2\lesssim \frac{1}{N^{2}}.
\end{equation}
From \e{x55}, \e{x15} and \e{x16} we obtain \e{x17}. Finally, having \e{x17}, \e{x5} and  \e{x99}, we get
\begin{align}
\|T^*(g)\|_{2}&\ge \|T^*(f)\|_{2}-\|T^*(g-f)\|_{2}\\
&\ge  \|T^*(f)\|_{2}-\sum_{k=1}^N\left\|T_{\Gamma_{\theta_k}}(g-f)\right\|_2\\
&\ge \|T^*(f)\|_{2}-c\\
&\gtrsim \log N\cdot  \|g\|_2.
\end{align}
This completes the proof of lemma.
\end{proof}

\section{A basic sequence of orthogonal functions}
In the sequel we always suppose $N$ to be a large enough integer. For the functions $f$ and $g$ introduced in the previous sections we will often use the relation
\begin{equation}\label{x99}
\|g\|_2\sim \|f\|_2\sim \sqrt {\log N}
\end{equation}
that easily follows from \e{x45} and \e{x17}.
\begin{lemma}\label{L5}
 Let $f\in L^2(\ZR)$ and $\supp f\subset B(0,\delta)$. Then for any direction $\theta$ and number $A\ge 2\delta$ it holds the inequality
	\begin{equation}\label{x51}
	\|T_{\Gamma_\theta}(f)\cdot \ZI_{B(A,\infty)}\|_2\lesssim \sqrt{ \frac{\delta}{A}}\cdot \|f\|_2.
	\end{equation}
\end{lemma}
\begin{proof}
	In light of \e{2-27} and the conditions of the lemma we have
\begin{equation*}
\|T_{\Gamma_\theta}(f)\cdot \ZI_{B(A,\infty)}\|_2=\frac{1}{2}\cdot \|H_{\theta}(f)\cdot \ZI_{B(A,\infty)}\|_2,
\end{equation*}	
 so it is enough to prove \e{x51} for the $H_\theta$ instead of the operator $T_{\Gamma_{\theta}}$. Without loss of generality we can suppose that $\theta=0$. So we have 
	\begin{equation}
	H_\theta f(\x)=\int_{\ZR}\frac{f(x_1-t,x_2)}{t}dt=\int_{-\delta}^\delta\frac{f(t,x_2)}{t-x_1}dt.
	\end{equation}
Observe that 
\begin{align}
&H_\theta f(\x)=0,\quad |x_2|>\delta,\\
&|H_\theta f(\x)|\lesssim \frac{1}{|x_1|}\cdot \int_{-\delta}^\delta|f(t,x_2)|dt,\quad |x_2|\le \delta,\quad |x_1|>1,6\delta.
\end{align}
Thus, using $A\ge 2\delta$ and a simple geometric argument, we obtain
\begin{align}
\|H_\theta(f)\cdot \ZI_{B(A,\infty)}\|_2^2&\lesssim 2\int_{-\delta}^\delta\int_{0,8A}^\infty \frac{1}{|x_1|^2}\cdot \left(\int_{-\delta}^\delta|f(t,x_2)|dt\right)^2dx_1dx_2\\
&\lesssim\frac{1}{A}\int_{-\delta}^\delta \left(\int_{-\delta}^\delta|f(t,x_2)|dt\right)^2dx_2\\
&\lesssim \frac{\delta}{A}\|f\|_2^2
\end{align}
and so \e{x51}.
\end{proof}
Denote
\begin{align*}
&S(\alpha,\beta)=\Gamma_\beta\setminus \Gamma_\alpha\\
&\qquad \quad \,=\{\textbf{x}\in \ZR^2:\,x_1\cos \beta +x_2\sin  \beta \ge 0,\, x_1\cos   \alpha +x_2\sin  \alpha<0\}
\end{align*}
that is a sectorial region. 
\begin{lemma}\label{L6}
	Let, $0<\delta<1/16$, $f\in L^2(\ZR)$ and $\supp f\subset B(0,\delta)$. Then for any directions $\alpha,\beta$ it holds the inequality
	\begin{equation}
	\|T_{S(\alpha, \beta)}(f)\cdot \ZI_{B(1/2,\infty)}\|_2\lesssim\sqrt[4]{\delta}\cdot \|f\|_2
	\end{equation}
\end{lemma}
\begin{proof}
	Observe that
	\begin{equation*}
	T_{S(\alpha, \beta)}=T_{\Gamma_\beta}\circ T_{\Gamma_{\pi+\alpha}}.
	\end{equation*}
	Consider the functions
	\begin{equation*}
	f_1=T_{\Gamma_{\pi+\alpha}}(f)\cdot \ZI_{B(0,\sqrt{\delta})},\quad f_2=T_{\Gamma_{\pi+\alpha}}(f)\cdot \ZI_{B(\sqrt{\delta},\infty)}.
	\end{equation*}
Applying \lem{L5} for $A=\sqrt{\delta}$, we obtain
	\begin{equation}\label{x56}
\|f_2\|_2=\|T_{\Gamma_{\pi+\alpha}}(f)\cdot \ZI_{B(\sqrt{\delta},\infty)}\|_2\lesssim \sqrt{ \frac{\delta}{\sqrt{\delta}}}\cdot \|f\|_2=\sqrt[4]{\delta}\cdot \|f\|_2
\end{equation}
and so
\begin{equation*}
\|T_{\Gamma_\beta}(f_2)\|_2\le \|f_2\|_2\lesssim  \sqrt[4]{\delta}\cdot \|f\|_2.
\end{equation*}
Once again apply \lem{L5} for $A=1/2$, we get
\begin{equation}\label{x57}
\|T_{\Gamma_\beta}(f_1)\cdot \ZI_{B(1/2,\infty)})\|_2\lesssim \sqrt[4]{\delta}\|f_1\|_2\le   \sqrt[4]{\delta}\|f\|_2.
\end{equation}
Finally, combining \e{x56} and \e{x57}, we obtain
\begin{align*}
	\|T_{S(\alpha, \beta)}(f)\cdot \ZI_{B(1/2,\infty)}\|_2&=\|T_{\Gamma_\beta}(T_{\Gamma(\pi+\alpha)}(f))\cdot \ZI_{B(1/2,\infty)}\|_2\\
	&\le \|T_{\Gamma_\beta}(f_1)\cdot \ZI_{B(1/2,\infty)}\|_2+\|T_{\Gamma_\beta}(f_2)\|_2\\
	&\lesssim  \sqrt[4]{\delta}\|f\|_2.
\end{align*}
\end{proof}

Denote 
\begin{align}
&S_k^+=S(\theta_k,\theta_{k-1}),\quad 
S_k^-=S(\theta_{k-1},\theta_{k}),\\
&S_k=S_k^+\cup S_k^-.
\end{align}
and consider the functions
\begin{equation}\label{x52}
g_k=T_{S_k^+}(g)-T_{S_k^-}(g),\quad k=1,2,\ldots,N,
\end{equation}
where $g$ is \e{x18}.
\begin{lemma}
	The sequence of functions \e{x52} satisfies the bound
	\begin{equation}\label{x66}
\left\|\max_{1\le m\le N}\left|\sum_{k=1}^mg_k\right|\right\|_2\gtrsim \log N\cdot \left\|\sum_{k=1}^Ng_k\right\|_2=\log N\cdot \|g\|_2
	\end{equation}
\end{lemma}
\begin{proof}
One can check that
	\begin{equation}\label{x37}
	T_{\Gamma_m}(g)=T_{\Gamma_0}(g)+\sum_{k=1}^mg_k,\quad  \left\|\sum_{k=1}^Ng_k\right\|_2=\|g\|_2.
	\end{equation}
So from \e{x7} we obtain 
\begin{align*}
\left\|\max_{1\le m\le N}\left|\sum_{k=1}^mg_k\right|\right\|_2&\ge \left\|\max_{1\le m\le N}\left|T_{\Gamma_0}(g)+\sum_{k=1}^mg_k\right|\right\|_2-\|T_{\Gamma_0}(g)\|_2\\
&= \left\|\max_{1\le m\le N}|T_{\Gamma_m}(g)|\right\|_2-\|T_{\Gamma_0}(g)\|_2\\
&\ge \left\|T^*(g)\right\|_2-\|g\|_2\\
&\gtrsim \log N\cdot \|g\|_2.
\end{align*}
\end{proof}

Now denote
\begin{align*}
&D_k^+=B(5N^4,N^{15})\cap S(\theta_k-1/N^4,\theta_{k-1}+1/N^4),\\
&D_k^-=B(5N^4,N^{15})\cap S(\theta_{k-1}+1/N^4, \theta_k-1/N^4),\\
&D_k=D_k^+\cup D_k^-.
\end{align*}
And consider the functions 
\begin{equation}\label{x58}
f_k=T_{D_k^+}(g)-T_{D_k^-}(g),\quad k=1,2,\ldots,N.
\end{equation}
\begin{lemma}
	We have the inequality
	\begin{equation}\label{x65}
	\|f_k- g_k\|_2\lesssim \|g\|_2/N^2,\quad k=1,2,\ldots,N. 
	\end{equation}
\end{lemma}
\begin{proof}
	First observe that, since $g$ and so $\hat g$ are a spherical functions, we have
	\begin{equation}\label{x59}
	\|T_{S(\alpha,\beta)}(g)\|_2^2=\|\hat g\cdot \ZI_{S(\alpha,\beta)}\|_2^2=\frac{|\alpha-\beta|}{2\pi }\cdot \|\hat g\|_2^2=\frac{|\alpha-\beta|}{2\pi }\cdot \|g\|_2^2.
	\end{equation}
	In view of \e{x6},  \e{x52} and \e{x58} it follows that 
	\begin{equation}
	\supp(\hat f_k-\hat g_k)\subset  B(0,5N^4)\cup \left(\cup_{j=1}^4U_j\right),
	\end{equation}
	where
	\begin{align*}
	&U_1=S(\theta_j-N^{-4},\theta_j),\, U_2=S(\theta_{j-1}+N^{-4}, \theta_{j-1}),\\ 
	&U_3=S(\theta_j,\theta_j-N^{-4}),\, U_4=S(\theta_{j-1}, \theta_{j-1}+N^{-4}).
	\end{align*}
	Besides, according to \e{x31} we have 
	\begin{equation}\label{x67}
	\|\hat g\|_\infty\le \|\hat f\|_\infty\cdot \|\hat K\|_\infty\lesssim \|\hat f\|_\infty \le \|f\|_1\lesssim N^{-8}.
	\end{equation}
	Thus, using \e{x59} and \e{x99}, we get
	\begin{align*}
	\|f_k- g_k\|_2&=	\|\hat f_k- \hat g_k\|_2\\
	&\le \|\hat g\cdot \ZI_{B(0,5N^4)}\|_2+\sum_{k=1}^4\|\hat g\cdot \ZI_{U_k}(g)\|_2\\
	&\lesssim \|\hat g\|_\infty \cdot N^4+N^{-2}\|g\|_2\\
	&\lesssim N^{-4} +N^{-2}\|g\|_2\\
	&\lesssim N^{-2}\|g\|_2.
	\end{align*}
\end{proof}
\begin{lemma} 
	It holds the inequality
	\begin{equation}\label{x41}
	\|f_k\cdot \ZI_{B(1/2,\infty)}\|_2\lesssim \frac{\|g\|_2}{N^2}.
	\end{equation}
\end{lemma}
\begin{proof}
	Letting
	\begin{equation}
	g_1=T_{B(5N^4,\infty)}(g),\quad g_2=T_{B(0,5N^4)}(g).
	\end{equation}
	We write 
	\begin{equation*}
	g=g_2+g_1=g_2+g_1\cdot \ZI_{B(0,N^{-8})}+g_1\cdot \ZI_{B(N^{-8},\infty)}=g_2+U+V.
	\end{equation*}
	Using \e{x6} and the definitions of the domains $D_k^+$ and $D_k^-$, one can write
	\begin{align}
	T_{D_k^+}(g)-T_{D_k^-}(g)&=T_{D_k^+}(g_1)-T_{D_k^-}(g_1)=T_{G_k^+}(g_1)-T_{G_k^-}(g_1)\\
	&=T_{G_k^+}(U)-T_{G_k^-}(U)+T_{G_k^+}(V)-T_{G_k^-}(V),\label{x68}
	\end{align}
	where 
	\begin{align*}
	G_k^+=S(\theta_k-1/N^4,\theta_{k-1}+1/N^4),\quad G_k^-=S(\theta_{k-1}+1/N^4, \theta_k-1/N^4).
	\end{align*}
	By \e{x67} we have 
	\begin{equation}
	\|g_2\|_2=\|\hat g\cdot \ZI_{B(0,5N^4)}\|_2\lesssim \frac{1}{N^8}\|\ZI_{B(0,5N^4)}\|_2\lesssim \frac{1}{N^4}.
	\end{equation}
	Combination of $\supp f\subset B(0,N^{-8})$ with inequality \e{x17} implies 
	\begin{align}
		\|T_{G_k^+}(V)-T_{G_k^-}(V)\|_2&\le \|V\|_2=\|(f-g_1)\cdot \ZI_{B(N^{-8},\infty)}\|_2\\
		&\le \|f-g_1\|_2\lesssim \|f-g\|_2+\|g_2\|_2\lesssim N^{-2}. \label{x53}
	\end{align}
	Then, applying \lem{L6} with $\delta=N^{-8}$, and taking into account that $\supp U\subset B(0,N^{-8})$ we obtain
	\begin{equation}\label{x54}
	\|(T_{G_k^+}(U)-T_{G_k^-}(U))\cdot \ZI_{B(1/2,\infty)}\|_2\lesssim \frac{\|U\|_2}{N^2}\le \frac{\|g\|_2}{N^2}.
	\end{equation}
	From \e{x58}, \e{x68}, \e{x53}, \e{x54} and \e{x99} we obtain 
	\begin{align*}
		\|f_k\cdot \ZI_{B(1/2,\infty)}\|_2&= \|(T_{D_k^+}(g)-T_{D_k^-}(g))\cdot \ZI_{B(1/2,\infty)}\|_2\\
		&\le  \|(T_{G_k^+}(V)-T_{G_k^-}(V))\cdot \ZI_{B(1/2,\infty)}\|_2\\
		&\qquad +\|(T_{G_k^+}(U)-T_{G_k^-}(U))\cdot \ZI_{B(1/2,\infty)}\|_2\\
		&\lesssim \|g\|_2/N^2
	\end{align*}
	and so \e{x41}.
\end{proof}

\begin{lemma}\label{L14}
	There exist a sequence of functions $r_k\in L^2(\ZR^2)$, $k=1,2,\ldots,N$, such that
	\begin{align}
	&\supp (r_k)\subset (-1/2,1/2)\times (-1/2,1/2),\label{x42}\\
	&\|\hat r_k\cdot \ZI_{\ZR^2\setminus D_k}\|_{2}\lesssim \|g\|_2/N^2,\label{x43}\\
	&\left\|\max_{1\le m\le N}\left|\sum_{k=1}^mr_k\right|\right\|_2\gtrsim \log N\left\|\sum_{k=1}^Nr_k\right\|_2\sim \log N\cdot \|g\|_2.\label{x44}
	\end{align}
\end{lemma}
\begin{proof}
Define 
\begin{equation*}
r_k(\x )=f_k(\x )\cdot \ZI_{B(0,1/2)}(\x ).
\end{equation*}
We will immediately have \e{x42}. From \e{x41} it follows that
\begin{equation*}
\|r_k-f_k\|_2\lesssim \|g\|_2/N^2,
\end{equation*}
and in light of \e{x65} we get $\|r_k-g_k\|_2\lesssim \|g\|_2/N^2$. Thus, taking into account \e{x66}, we get \e{x44}. Since
\begin{equation*}
r_k=f_k-f_k\cdot \ZI_{B(1/2,\infty)},
\end{equation*}
and $\supp \hat f_k\subset D_k$, by \e{x41} and \e{x99} we get
\begin{align}
\|\hat r_k\cdot \ZI_{\ZR^2\setminus D_k}\|_{2}&=\|\widehat{f_k\cdot \ZI}_{B(1/2,\infty)}\cdot \ZI_{\ZR^2\setminus D_k}\|_{2}\le \|\widehat{f_k\cdot \ZI}_{B(1/2,\infty)}\|_{2}\\
&= \|f_k\cdot \ZI_{B(1/2,\infty)}\|_{2}\lesssim \| f\|_2/N^2
\end{align}
and so \e{x43}.
\end{proof}
\section{Double trigonometric polynomials}

The following lemma is a version of \lem{L14} for double trigonometric sums.
\begin{proposition}
	There exist two dimensional non-overlapping trigonometric polynomials 
	\begin{equation}\label{x61}
	p_k(\textbf{x})=\sum_{\textbf{n}\in G_k} a_{\textbf{n}}e^{2\pi i\textbf{n}\cdot \textbf{x}},\quad k=1,2,\ldots, N,
	\end{equation}
	such that 
	\begin{align}
	&G_k\subset B(0,4N^{15})\cap \ZZ_+^2,\label{x24}\\
	&\left\|\max_{1\le m\le N}\left|\sum_{k=1}^mp_k\right|\right\|_2\gtrsim \log N\left\|\sum_{k=1}^Np_k\right\|_2.\label{x25}
	\end{align}
\end{proposition}
\begin{proof}
Let $\textbf{u}$ be a fixed vector. In light of \e{x42} the function 
\begin{equation}\label{u16}
r_k(\textbf{u},\x)=e^{-2\pi i\textbf{u}\cdot \x}r_k(\x)
\end{equation}
depended on $\x$ can be periodically continued and considered as a function of $L^2(\ZT^2)$ with the Fourier representation
\begin{equation*}
r_k(\textbf{u},\x)=\sum_{n\in  \ZZ^2} \hat r_k(\n+\textbf{u})e^{2\pi i\n\cdot \x}.
\end{equation*}
 For any $\n=(n_1,n_2)\in \ZZ^2$ we denote $\Delta_{\n}=[n_1,n_1+1)\times [n_2,n_2+1)$ and let
\begin{equation}
U_k=\{\n\in \ZZ^2:\, \Delta_\n\cap D_k\neq\varnothing\}\subset B(0,2N^{15}).
\end{equation}
From the definition of $D_k$ it follows that
\begin{align}
D_k\subset \bigcup_{\n\in U_k}\Delta_\n\subset B(0,2N^{15}).\label{u15}
\end{align}
A simple geometric argument shows that $\dist (D_k,\ZR^2\setminus S_k)>2$ that implies
\begin{equation*}
\bigcup_{n\in U_k}\Delta_n\subset S_k=S_k^+\cup S_k^-,
\end{equation*}
so $U_k$ are pairwise disjoint. Consider the functions
\begin{align*}
&p_k(\textbf{u},\x)=\sum_{\n\in  U_k} \hat r_k(\n+\textbf{u})e^{2\pi i\n\cdot \x},\\
&q_k(\textbf{u},\x)=\sum_{\n\in \ZZ^2\setminus U_k} \hat r_k(\n+\textbf{u})e^{2\pi i\n\cdot \x}.
\end{align*}
For a fixed $\u$ the polynomials $p_k(\textbf{u},\x)$ are non-overlapping, since $U_k\subset S_k$. Clearly,
\begin{equation}\label{x60}
r_k(\textbf{u},\x)=p_k(\textbf{u},\x)+q_k(\textbf{u},\x)
\end{equation}
and by \e{x43} and \e{u15} we obtain
\begin{align*}
\int_{\ZT^2}\|q_k(\textbf{u},\cdot)\|_2^2d\textbf{u}&=\int_{\ZT^2}\sum_{n\in \ZZ^2\setminus U_k} |\hat r_k(\n+\textbf{u})|^2d\textbf{u} =\sum_{\n\in \ZZ^2\setminus U_k}\int_{\Delta_\n}|\hat r_k(t)|^2dt\\
&\le \|\hat r_k\cdot \ZI_{\ZR^2\setminus D_k}\|_{2}^2\lesssim \frac{\|g\|_2^2}{N^4}
\end{align*}
and so
\begin{equation*}
\sum_{k=1}^N\int_{\ZT^2}\|q_k(\textbf{u},\cdot)\|_2^2d\textbf{u}\lesssim \frac{\|g\|_2^2}{N^3}.
\end{equation*}
This inequality produces a $\textbf{u}=\textbf{u}_0$ such that
\begin{equation*}
\sum_{k=1}^N\|q_k(\textbf{u}_0,\cdot )\|_2^2\lesssim \frac{\|g\|_2^2}{N^3}.
\end{equation*}
and by H\"{o}lder inequality we get
\begin{equation}\label{x69}
\sum_{k=1}^N\|q_k(\textbf{u}_0,\cdot)\|_2\le \sqrt N\left(\sum_{k=1}^N\|q_k(\textbf{u}_0,\cdot )\|_2^2\right)^{1/2}\lesssim \frac{\|g\|_2}{N}.
\end{equation}
Finally, we can define polynomials
\begin{equation*}
p_k(\x)=e^{2\pi i\cdot 2(N^{15}x_1+ N^{15}x_2)}p_k(\textbf{u}_0,\x),\quad \x=(x_1,x_2),
\end{equation*} 
with the non-overlapping spectrums
\begin{equation*}
G_k=U_k+(2N^{15},2N^{15})\subset B(0,4N^{15})\cap \ZZ_+^2.
\end{equation*}
Combining \e{x44}, \e{u16}, \e{x60} and \e{x69}, we get
\begin{align}
\left\|\max_{1\le m\le N}\left|\sum_{k=1}^mp_k\right|\right\|_2&=\left\|\max_{1\le m\le N}\left|\sum_{k=1}^mp_k(\u_0,\cdot )\right|\right\|_2\\
&\ge \left\|\max_{1\le m\le N}\left|\sum_{k=1}^mr_k(\u_0,\cdot )\right|\right\|_2-\sum_{k=1}^N\|q_k(\textbf{u}_0,\cdot)\|_2\\
&\ge  \left\|\max_{1\le m\le N}\left|\sum_{k=1}^mr_k\right|\right\|_2-c_2\cdot \frac{\|g\|_2}{N}\\
&\ge c_1 \log N\left\|\sum_{k=1}^Nr_k\right\|_2-c_2\cdot \frac{\|g\|_2}{N}.\label{x70}
\end{align}
Likewise, using \e{x99}, one can show 
\begin{equation}\label{x71}
\left|\left\|\sum_{k=1}^Nr_k\right\|_2-\left\|\sum_{k=1}^Np_k\right\|_2\right|\lesssim \frac{\|g\|_2}{N}.
\end{equation}
From \e{x44}, \e{x70} and \e{x71} one can easily get \e{x25}.
\end{proof}

\section{Equivalency of discrete trigonometric systems}
Let $\{f_k:\,k\in A\}$ and $\{g_k:\, k\in B\}$ be families of measurable complex-valued functions defined on
measure spaces $(X,\mu)$ and $(Y,\nu)$ respectively. We say these sequences are equivalent if there is a one to one mapping $\sigma:A\to B$ such that the equality 
\begin{equation}
\mu\{f_{\alpha_j}\in B_j,\, j=1,2,\ldots,m\}=\nu\{g_{\sigma(\alpha_j)}\in B_j,\, j=1,2,\ldots,m\}
\end{equation}
holds for any choice of indexes $\alpha_j\in A$ and open balls $B_j\subset  \ZR^2$, $j=1,2,\ldots,m$. For an integer $l\ge 1$ we denote $\ZN_l=\{1,2,\ldots ,l\}$.
The discrete trigonometric system of order $l$ on $[0,1)$ is defined as
\begin{equation*}
T^{(l)}=\left\{t_n^{(l)}(x)=\sum_{k=1}^l\exp\left(2\pi i \frac{nk}{l}\right)\cdot \ZI_{\delta_k^{(l)}}(x),\quad n\in \ZN_l\right\},
\end{equation*}
where $\delta_k^{(l)}=[(k-1)/l,k/l)$. The tensor product of two one dimensional systems of order $p$ and $q$ is the collection of functions
\begin{align*}
T^{(p)}&\times T^{(q)}\\
&=\left\{(t_{n_1}^{(p)}\times t_{n_2}^{(q)})(\x)=t_{n_1}^{(p)}(x_1)\cdot t_{n_2}^{(q)}(x_2),\,\n\in \ZN_p\times \ZN_q \right\}.
\end{align*}
Noice that
\begin{align}
(t_{n_1}^{(p)}\times t_{n_2}^{(q)})(\x)=\sum_{u_1=1}^p\sum_{u_2=1}^q\exp2\pi i\left( \frac{n_1u_1}{p}+\frac{n_2u_2}{q}\right)\cdot \ZI_{\delta_{u_1}^{(p)}\times \delta_{u_2}^{(q)}}(\x).\label{u17}
\end{align}
We prove the following
\begin{proposition}\label{P1}
If $p$ and $q$ are coprime numbers, then the  systems $T^{(pq)}$ and $T^{(p)}\times T^{(q)}$ are equivalent.
\end{proposition}
\begin{lemma}\label{L4}
	For any coprime numbers $p$ and $q$, there are two one to one mappings $\phi$ and $\psi$ acting from $\ZN_p\times\ZN_q$ to $\ZN_{pq}$ such that
	\begin{equation}\label{x26}
	\left\{\frac{n_1u_1}{p}+\frac{n_2u_2}{q}\right\}=\left\{\frac{\phi(\n)\psi(\u)}{pq}\right\},
	\end{equation}
where $\{a\}$ denotes the fractional part of a real number $a$.
\end{lemma}
\begin{proof}
	According to Chinese reminder theorem for every pair $(n_1,n_2)$ of integers $n_1\in \ZN_p$ and $n_2\in \ZN_q$ one can find a unique $l\in \ZN_{pq}$ such that
	\begin{align*}
	l=n_1\mod p,\\
	l=n_2\mod q,
	\end{align*}
	and this defines a one to one mapping $\tau$ from $\ZN_p\times\ZN_q$ to $\ZN_{pq}$ such that $\tau(n_1,n_2)=l$. For our further convenience we extend the mapping $\tau $ over whole $\ZZ^2_+$ periodically by $\tau(n_1,n_2)=\tau(n_1+p\cdot k,n_2+q\cdot j)$ satisfied for any pair of integers $k,j$.  Define
	\begin{align*}
	&\phi(\n)=\tau(n_1,n_2),\\
	&\psi(\u)=\tau(u_1,-u_2)\cdot(q-p)\mod^* pq,
	\end{align*}
	where 
\begin{equation}
m\mod^*n=\left\{\begin{array}{lrl}
n&\hbox{ if }& m\mod n=0,\\
m\mod n&\hbox{ if }&m\mod n\neq 0 .
\end{array}
\right.
\end{equation}	
	Clearly $\phi$ and $\psi$ determine one to one mappings from $\ZN_p\times\ZN_q$ to $\ZN_{pq}$. Besides we have
	\begin{align*}
	\left\{\frac{\phi(\n)\psi(\u)}{pq}\right\}&=	\left\{\frac{\tau(n_1,n_2)\tau(u_1,-u_2)(q-p)}{pq}\right\}\\
	&=	\left\{\frac{\tau(n_1,n_2)\tau(u_1,-u_2)}{p}-\frac{\tau(n_1,n_2)\tau(u_1,-u_2)}{q}\right\}\\
	&=\left\{\frac{n_1u_1}{p}+\frac{n_2u_2}{q}\right\}.
	\end{align*}
\end{proof}
\begin{proof}[Proof of \pro {P1}]
Let $\phi$ and $\psi$ be mappings taken from \lem{L4}. The mapping $\phi$ produces a one to one correspondence 
\begin{equation}
t_{n_1}^{(p)}\times t_{n_2}^{(q)}\to t_{\phi(\n)}^{(pq)},
\end{equation}
while $\psi$ produces
\begin{equation}
\delta_{u_1}^{(p)}\times \delta_{u_2}^{(q)}\to \delta_{\psi(u_1,u_2)}^{(pq)}.
\end{equation}
In light of \e{u17} and \e{x26} one can see that each function $t_{n_1}^{(p)}\times t_{n_2}^{(q)}$ takes the same value on $\delta_{u_1}^{(p)}\times \delta_{u_2}^{(q)}$ as $t_{\phi(n_1,n_2)}^{(pq)}$ on $\delta_{\psi(u_1,u_2)}^{(pq)}$. This obviously implies the equivalency of the systems $T^{(pq)}$ and $T^{(p)}\times T^{(q)}$.
\end{proof}
\section{The main lemma}
\begin{lemma}\label{L18}
	If functions $f,g\in L^2(\ZT)$ satisfy strong orthogonality condition
	\begin{equation}\label{x64}
	 \int_\ZT f(x)g(h-x)dx=0 \text { for any }h\in \ZR,
	\end{equation}
	then they have non-overlapping Fourier series:
\end{lemma}
\begin{proof}
The condition \e{x64} implies 
	\begin{equation*}
\hat f(n)\cdot \hat g(n)=(\widehat{ f\star g})(n)=0.
	\end{equation*}
So for each $n\in \ZZ^2$ either $\hat f(n)=0$ or $\hat g(n)=0$. This completes the proof.
\end{proof}
Now we are able to proof the main lemma.
\begin{lemma}\label{L10}
	There exist a sequence of one dimensional trigonometric polynomials 
	\begin{align*}
	Q_k(x)=\sum_{n\in U_k}a_ne^{2\pi i nx},\quad k=1,2,\ldots ,N,
	\end{align*}
	with non-overlapping spectrums $U_k$ such that 
	\begin{align}
		&U_k\subset [1,N^{70}],\label{x62}\\
	&\left\|\max_{1\le m\le N}\left|\sum_{k=1}^mQ_k\right|\right\|_2\gtrsim \log N\left\|\sum_{k=1}^NQ_k\right\|_2.\label{x63}
	\end{align}
\end{lemma}
\begin{proof}
 For the coprime numbers $p=N^{31}$ and $q=N^{31}+1$ we consider the discrete double trigonometric system $T^{(p)}\times T^{(q)}$. It is easy to see that 
 \begin{equation}\label{u18}
 \left| e^{2\pi i\textbf{n}\cdot \textbf{x}}-(t_{n_1}^{(p)}\times t_{n_2}^{(q)})(\x)\right|\lesssim \frac{1}{N^{16}},\quad \x\in \ZT^2,\, \textbf{n}=(n_1,n_2)\in B(0,4N^{15})\cap \ZZ_+.
 \end{equation}
 Consider the non-overlapping double discrete trigonometric polynomials
 \begin{equation}
 P_k(\x)=\sum_{\n\in G_k}a_{\textbf{n}}(t_{n_1}^{(p)}\times t_{n_2}^{(q)})(\x),\quad k=1,2,\ldots, N,
 \end{equation}
 with the same coefficients as we have in \e{x61}, where $G_k\subset \ZN_p\times\ZN_q$. According to \pro{P1} the systems $T^{(pq)}$ and $T^{(p)}\times T^{(q)}$ are equivalent. So the sequence of double polynomials $P_k$ generates a one dimensional sequence of non-overlapping polynomials $R_k\in T^{(pq)}$, $k=1,2,\ldots,N$.  Both sequences share the same logarithmic bound \e{x25} of $p_k$, since from \e{u18} it follows that
 \begin{equation*}
 \| P_k-p_k\|_2\lesssim \frac{1}{N}\left(\sum_{\n\in G_k}a_{\n}^2\right)^{1/2}\le \frac{1}{N}\left\|\sum_{k=1}^Np_k\right\|_2.
 \end{equation*}
  Disjointness of the spectrums of $R_k$ as polynomials of $T^{(pq)}$ implies 
 \begin{equation}
 \int_\ZT R_k(x)R_m(h-x)dx=0 \text { for any }h\in \ZR,\,k,m\in \ZN_{pq},\, k\neq m.
 \end{equation}
 Thus, according to \lem{L18}, the functions $R_k\in L^2(\ZT)$ have non-overlapping spectrums of Fourier series. Besides, they are step functions with the intervals of constancy having the length $(pq)^{-1}\sim N^{-62}$. Let  
 \begin{equation*}
 Q_k(x)=e^{2\pi i (N^{66}+1)x}\cdot \sigma_{N^{66}}(x,R_k),
 \end{equation*}
 where $\sigma_{n}(x,f)$ denotes the $n$ order $(C,1)$ mean of a function $f$. Cleraly we will have \e{x62}. Recall the approximation property of the $(C,1)$  means 
 \begin{align}
 \|\sigma_n(f)-f\|_2&\le \left(\int_\ZT\int_\ZT K_n(t)|f(x+t)-f(x)|^2dtdx\right)^{1/2}\\
 &\le \left(\int_{-\delta}^\delta K_n(t)\int_\ZT|f(x+t)-f(x)|^2dxdt\right)^{1/2}\\
 &\qquad +\left(\int_{\delta<|t|<\pi}K_n(t)\int_\ZT |f(x+t)-f(x)|^2dxdt\right)^{1/2}\\
 &\lesssim \omega_2(\delta,f) +\|f\|_2\left(\int_\delta^\infty \frac{1}{nt^2}dt\right)^{1/2}\\
 &\lesssim \omega_2(\delta,f) +\frac{\|f\|_2}{\sqrt{n\delta}}.
 \end{align}
Using this inequality for $f=R_k$, $n=N^{66}$ and $\delta=N^{-64}$, as well as the easily checked bound $\omega_2(\delta,R_k)\lesssim \|R_k\|_2/N$, one can obtain
 \begin{equation*}
 \left\|Q_k(x)-e^{2\pi i (N^{66}+1)x}\cdot R_k(x)\right\|_2\lesssim\frac{ \|R_k\|_2}{N}.
 \end{equation*}
 The latter immediately yields the logarithmic bound \e{x63}, since we have the same bound for $R_k$. Lemma is proved.

\end{proof}

\section{Proof of the main Theorem and \cor{C0}}
\begin{lemma}\label{L11}
	Let $T$ be a sublinear operator satisfying 
	\begin{align*}
	&\|T\|_{L^2\to L^{2,\infty}}\le c\sqrt {\log N},\\
	&\|T\|_{L^2\to L^{\infty}}\le  N
	\end{align*}
	and $c\log N\ge 1$. Then we have
	\begin{equation}
	\|T\|_{L^2\to L^{2}}\lesssim c\log N.
	\end{equation}
\end{lemma}
\begin{proof}
For a given function $f\in L^2(\ZT)$, $\|f\|_2\le 1$, we have $\|T(f)\|_\infty\le N$. Denote
\begin{align*}
\phi(\lambda)=|\{x:\, |Tf(x)|>\lambda\}|.
\end{align*}
We have 
\begin{align}
&\phi(\lambda)=0\text { if } \lambda>N,\\ 
&\phi(\lambda)\le \frac{\|T\|_{L^2\to L^{2,\infty}}^2}{\lambda^2}\le \frac{c^2\log N}{\lambda^2},\quad \lambda>0,
\end{align} 
and so
\begin{align}
\|T(f)\|_2^2=2\int_0^\infty \lambda \phi(\lambda)d\lambda=2\int_0^N \lambda \phi(\lambda)d\lambda\le 2+2\int_1^N\lambda \phi(\lambda)d\lambda\lesssim c^2\log^2N.
\end{align}

\end{proof}
\begin{proof}[Proof of \trm{T3}]
The upper bound of \e{a50} follows from the Menshov-Rademacher inequality. The lower bound  
\begin{equation}
\max_\sigma\|T_{\sigma,N}\|_{L^2\to L^2}\gtrsim  \log N
\end{equation}
easily follows from \lem{L10}. 
\end{proof}
\begin{proof}[Proof of \cor{C0}]
	Combination of \lem{L11} and lower bound of \e{a50} implies \e{a51}.
\end{proof}
\section{Proof of \cor{C1}}
The next lemma based on the inequality \e{a51}. Denote by $\ZP_N$ the family of one dimensional trigonometric polynomials of the form
\begin{equation*}
\sum_{k=1}^Na_ke^{2\pi ik},
\end{equation*}
where $a_k$ are complex numbers.
\begin{lemma}\label{L12}
	For any $N>N_0$ there exists a polynomial $P\in \ZP_N$ and a rearrangement $\sigma\in \Sigma_N$ such that $\|P\|_2\sim 1$ and 
	\begin{equation}\label{u23}
	\left|\left\{x\in \ZT:\, T_{\sigma, N}(x, P)>\sqrt {\log N}\right\}\right|\gtrsim 1.
	\end{equation}
\end{lemma}
\begin{proof}
	Let $M=[\sqrt{N}]+1$. According to \e{a51} there is a polynomial $Q\in \ZP_M$ with $\|Q\|_2=1$ and a rearrangement $\tau\in \Sigma_M$ such that the inequality
	\begin{equation}
	|E|=\left|\left\{x\in \ZT:\,T_{\tau, M}(x, Q)>\lambda_0\right\}\right|\ge\frac{c\log M}{\lambda_0^2},
	\end{equation}
holds for some $\lambda_0>0$. Since $|\ZT|=1$, we have $\lambda_0\ge \sqrt{c\log M}$ and from $\|Q\|_2=1$ it follows that $0<\lambda_0\le \sqrt{M}$. Put $l=[\lambda_0^2/c\log M]$, we have 
\begin{equation}
1\le l\le \frac{M}{c\log M}\le \frac{M}{2},\quad |E|>\frac{1}{l+1},
\end{equation}
for $N>N_0$. By a well-known argument (see \cite{Zyg}, ch. 13, Lemma 1.24) one can find a sequence of points $x_k\in \ZT$, $k=0,1,\ldots,l-1$, such that
\begin{equation}
|F|=\left|\bigcup_{k=0}^{l-1}(E+x_k)\right|\ge 1-(1-|E|)^l\ge 1-(1-{(l+1)}^{-1})^l\gtrsim 1.
\end{equation}
Then we consider the polynomial
\begin{equation}
G(x)=\frac{1}{\sqrt l}\sum_{k=0}^{l-1}Q_k(x),\text{ where }  Q_k(x)=Q(x-x_k)e^{2\pi ikMx}.
\end{equation}
Clearly $G\in\ZP_{lM}\subset  \ZP_N$, since $lM\le N$. Define a rearrangement $\sigma\in \Sigma_N$ by
\begin{align*}
&\sigma(n)=\tau(n-Mk)+Mk\text{ if }Mk<n\le M(k+1),\quad k=0,1,\ldots,l-1,\\
&\sigma(n)=n \text { if }lM<n\le N.
\end{align*}
One can check that $\|G\|_2=1$. Any partial sum of the $\sigma$-rearrangement of the polynomial $\frac{Q_k}{\sqrt l}$ can be written as a difference of two partial sums of $G$. This implies 
\begin{align}
T_{\sigma, N}(x, G)\ge \frac{1}{2\sqrt l}\cdot T_{\tau, M}(x-x_k, Q),\quad x\in \ZT.
\end{align}
Thus for any $x\in E+x_k$ we have
\begin{align}
T_{\sigma, N}(x, G)\ge \frac{1}{2\sqrt l}\cdot T_{\tau, M}(x-x_k, Q)> \frac{\lambda_0}{2\sqrt{l}}\ge \frac{\sqrt{c\log M}}{2}\gtrsim  \sqrt {\log N}.
\end{align}
Thus we get $T_{\sigma, N}(x, G)\gtrsim \sqrt {\log N}$ whenever $x\in F$. Since $|F|\gtrsim 1$, a polynomial $P(x)=c\cdot G(x)$ with a suitable absolute constant $c$ may become our desired polynomial.
\end{proof}
\begin{proof}[Proof of \cor{C1}]
	Using \e{y1} one can define integers $N_k\ge 1$, $k=1,2,\ldots$, such that
	\begin{equation}
	N_{k+1}> 2N_{k},\quad w(2N_k)\le \frac{\log N_k}{k^2},\quad k=1,2,\ldots.
	\end{equation}
	Applying \lem{L12}, we find polynomials $P_k\in \ZP_{N_k}$ and rearrangements $\sigma_k\in \Sigma_{N_k}$ such that $\|P_k\|_2\sim 1$ and the sets
	\begin{equation}
	E_k=\left\{x\in \ZT:T_{\sigma_k}(x, P_k)>\sqrt {\log N_k}\right\}
	\end{equation}
	satisfy $|E_k|>c>0$. It is well known this condition provides a sequence $t_k\in \ZT$ such that 
	\begin{equation}
	\left|\bigcap_{k\ge 1}\bigcup_{n\ge k}(E_n+t_n)\right|=1
	\end{equation}
	 (see \cite{Zyg}, ch. 13, Lemma 1.24). Consider the trigonometric series
	\begin{equation}\label{u19}
	\sum_{k=1}^\infty \frac{1}{k\cdot \sqrt{w(2N_k)}}\cdot P_k(x-t_k)e^{2\pi iN_k x}=\sum_{n=1}^\infty c_ne^{2\pi i nx}.
	\end{equation}
	We have 
	\begin{equation}\label{u20}
	\sum_{n=1}^\infty |c_n|^2w(n)\le \sum_{k=1}^\infty \frac{\|P_k\|_2^2}{k^2}<\infty.
	\end{equation}
	Define a permutation $\sigma$ of $\ZN$ as follows 
	\begin{equation}
	\sigma(n)=\left\{\begin{array}{lcl}
	&\sigma_k(n-N_k)+N_k&\hbox{ if }N_k<n\le 2N_k,\, k=1,2,\ldots,\\
	& n &\hbox{ if }n\notin \cup_{k\ge 1}(N_k,2N_k].
	\end{array}
	\right.
	\end{equation}
		If $x\in \cap_{k\ge 1}\cup_{n\ge k}(E_n+t_n)$, then $x\in E_k+t_k$ for infinitely many $k$. For $x\in E_k+t_k$ we have
	\begin{equation*}
	\max_{N_k<m \le 2N_k}\left|\sum_{n=N_k+1}^mc_{\sigma(n)}e^{2\pi i \sigma(n)x} \right|\ge \frac{T_{\sigma_k}(x-t_k, P_k)}{k\cdot \sqrt{w(2N_k)}}\ge \frac{\sqrt{\log N_k}}{k\cdot \sqrt{w(2N_k)}}>1.
	\end{equation*}
Thus we get that series \e{u19} is almost everywhere divergent. Combining this with \e{u20} we complete the proof.
\end{proof}

\section{Proof of \cor{C2}}

\begin{lemma}\label{L21}
	Let $f_k$, $k=1,2,\ldots,n$, be a sequence of complex valued functions on an interval $\Delta$. Then for any $\lambda>0$ we have inequality
	\begin{align}
	&\left|\left\{x\in \Delta: \Re\left(\max_{1\le m\le n}\sum_{k=1}^m\alpha f_k(x)\right)>\lambda/\sqrt 2 \right\}\right|\\
	&\qquad\qquad\qquad\ge \frac{1}{4}\cdot\left|\left\{x\in \Delta: \max_{1\le m\le n}\left|\sum_{k=1}^mf_k(x)\right|>\lambda \right\}\right|,\label{u26}
	\end{align}
	for some of four numbers $\alpha=e^{\frac{\pi ik}{2}}$, $k=0,1,2,3$.
\end{lemma}
\begin{proof}
	One can check that
	\begin{equation*}
	\max_{1\le m\le n}\left|\sum_{k=1}^mf_k(x)\right|>\lambda
	\end{equation*}
	yields at least one of the following four inequalities
	\begin{align*}
	\Re\left(\max_{1\le m\le n}\sum_{k=1}^m f_k(x)\right)>\lambda/\sqrt 2,\\
	\Re\left(\max_{1\le m\le n}\sum_{k=1}^m f_k(x)\right)<-\lambda/\sqrt 2,\\
	\Im\left(\max_{1\le m\le n}\sum_{k=1}^m  f_k(x)\right)>\lambda/\sqrt 2,\\
	\Im\left(\max_{1\le m\le n}\sum_{k=1}^m  f_k(x)\right)<-\lambda/\sqrt 2.
	\end{align*}
This immediately gives \e{u25} for some $\alpha$.
\end{proof}
The spectrum of a trigonometric polynomial
\begin{equation*}
U(x)=\sum_{k=m}^na_k e^{2\pi ikx}
\end{equation*}
will be denoted by $\spec(U)=\{k:\, a_k\neq 0\}$.
\begin{lemma}\label{L13}
	Let $\Delta\subset \ZT$ be an arbitrary interval and the integer $N>N_0$ satisfy 
	\begin{equation}\label{u42}
	\frac{1}{\sqrt N}\le |\Delta| 
	\end{equation}Then for any integer $l$ there exists a sequence of non-overlapping trigonometric polynomial $U_n$, $n=1,2,\ldots,N$, such that 
	\begin{align}
	&\spec(U_n)\subset (l,l+N^5],\\
	&\left\|\sum_{n=1}^NU_n\right\|_{L^2(\ZT)}\lesssim\sqrt{ |\Delta|},\\
	&\sum_{n=1}^N|U_n(x)|\lesssim 1/N,\quad x\in \ZT\setminus \Delta,\\
	&\left|\left\{x\in \Delta:\max_{1\le m\le N}\Re\left(\sum_{n=1}^mU_n(x)\right)>\sqrt {\log N}\right\}\right|\gtrsim |\Delta|.\label{u24}
	\end{align}
\end{lemma}
\begin{proof}
Suppose $\Delta=[a,b]$. In view of \e{u42} we consider the polynomial
\begin{equation}
R(x)=R_N(x)=\frac{1}{\pi}\int_{a+1/4\sqrt{N}}^{b-1/4\sqrt{N}} K_{[N^3/3]}(x-t)dt=\sum_{k=-[N^3/3]}^{[N^3/3]}c_ke^{2\pi ikx},
\end{equation}	
where $K_{n}$ is the Fej\'{e}r kernel of order $n$. Standard properties of the Fej\'{e}r kernel implies
\begin{align}
&0\le R(x)\lesssim 1/N^2,\quad x\in \ZT\setminus \Delta,\label{u43}\\
&1\ge R(x)\ge 1/2,\quad x\in \tilde \Delta=[a+1/3\sqrt{N},b-1/3\sqrt{N}],\label{u44}\\
&1\ge R(x)\ge 0,\quad x\in \Delta\setminus \tilde \Delta,
\end{align}
where \e{u44} will be satisfied for bigger enough $N$. Applying \lem{L12}, we find a polynomial
\begin{align*}
P(x)=\sum_{n=1}^Nb_ne^{2\pi inx},
\end{align*}
satisfying the conditions of lemma. In light of \e{u42}, from \e{u23} it easily follows that
\begin{equation}\label{u56}
\left|\left\{x\in \tilde \Delta:\, \max_{1\le m\le N}\left|\sum_{n=1}^mb_{\sigma(n)}e^{2\pi i\sigma(n)N^3x}\right|>\sqrt {\log N}\right\}\right|\gtrsim |\Delta|,
\end{equation}
where $\sigma$ is the permutation from \e{u23}. Consider the polynomials
\begin{equation}\label{u21}
p_n(x)=4e^{2\pi ilx}R(x)\cdot b_{\sigma(n)}e^{2\pi i\sigma(n)N^3x},\quad n=1,2,\ldots, N, 
\end{equation}
whose spectrums are in $(l,l+N^5]$ and non-overlapping. Using \e{u43} and \e{u44}, we conclude
\begin{equation*}
\left\|\sum_{n=1}^Np_n\right\|_{L^2(\ZT)}\lesssim \left\|P(N^3x)\right\|_{L^2(\Delta)}+\frac{1}{N^2}\lesssim\|P\|_2\cdot \sqrt{ |\Delta|}\lesssim \sqrt{ |\Delta|}.
\end{equation*}
Using \e{u43} and $\sum_{n=1}^N|b_n|\le \sqrt{N}\cdot \|P\|_2\lesssim \sqrt{N}$, we get
\begin{align*}
\sum_{n=1}^N|p_n(x)|\lesssim \sqrt{N} |R(x)|\lesssim \frac{1}{N}\text { for all }x\in \ZT\setminus \Delta.
\end{align*}
Then we have
\begin{equation*}
\max_{1\le m\le N}\left|\sum_{n=1}^mp_n(x)\right|=4R(x)\max_m\left|\sum_{n=1}^mb_{\sigma(n)}e^{2\pi i\sigma(n)N^3x}\right|,
\end{equation*}
and therefore by \e{u44} and \e{u56} we get
\begin{align}
&\left|\left\{x\in \Delta:\,\max_{1\le m\le N}\left|\sum_{n=1}^mp_n(x)\right|>2\sqrt {\log N}\right\}\right|\\
&\qquad\ge \left|\left\{x\in \tilde \Delta:\, \max_{1\le m\le N}\left|\sum_{n=1}^mb_{\sigma(n)}e^{2\pi i\sigma(n)N^3x}\right|>\sqrt {\log N}\right\}\right|\\
&\qquad\gtrsim |\Delta|.
\end{align}
According to \lem{L21} as a desired sequence can serve the polynomials $U_n(x)=e^{\frac{\pi ik}{2}}\cdot p_n(x)$ with an appropriate choice of $k=0,1,2,3$. Clearly they satisfy the conditions of the lemma.
\end{proof}
In the rest of the paper we will consider the sequence
\begin{equation}\label{u51}
\nu_0=1,\quad \nu_k=2^{50^k},\quad k=1,2,\ldots.
\end{equation}
\begin{lemma}\label{L20}
	If an increasing sequence of numbers $w(n)$ satisfies \e{4}, then there exists a set of integers $G\subset \ZN$ such that 
		\begin{align}
	&w(\nu_{k+1})<100w(\nu_{k}),\quad k\in G,\label{u48}\\
	&\sum_{k\in G} \frac{50^k}{w(\nu_k)}=\infty,\label{u49}
	\end{align}
	where $\nu_k$ is the sequence \e{u51}.
\end{lemma}
\begin{proof}
	First, observe that from \e{4} it follows that
	\begin{equation}\label{u57}
	\sum_{k=1}^\infty \frac{50^k}{w(\nu_k)}=\infty.
	\end{equation}
	Let $G$ be the set of integers $k$ satisfying \e{u48}. If 
	\begin{equation}\label{u50}
	\sum_{k\in \ZN\setminus G} \frac{50^k}{w(\nu_k)}<\infty,
	\end{equation}
	then \e{u49}  immediately follows from \e{u57} and lemma will be proved. So we can suppose that the series in \e{u50} is divergent. Clearly $G$ is infinite and the set $\ZN\setminus G$ can be written in the form
	\begin{equation*}
	\ZN\setminus G=\bigcup_{j} \{m_{2j}+1,m_{2j}+2,\ldots,m_{2j+1}\}
	\end{equation*}
	where $m_{2j}\in G$ for any $j=1,2,\ldots$. We have 
	\begin{equation*}
	w(\nu_{k+1})\ge 100w(\nu_{k})\text { for all }k=m_{2j}+1,m_{2j}+2,\ldots,m_{2j+1},
	\end{equation*}
	that implies 
	\begin{equation*}
	\sum_{k=m_{2j}+1}^{m_{2j+1}} \frac{50^k}{w(\nu_k)}\le  \frac{50^{m_{2j}+1}}{w(\nu_{m_{2j}+1})} \left(\frac{1}{2}+\frac{1}{2^2}+\ldots\right)\le\frac{50^{m_{2j}+1}}{w(\nu_{m_{2j}})}.
	\end{equation*}
	Thus we get
	\begin{equation*}
	\sum_{k\in G} \frac{50^k}{w(\nu_k)}\ge \sum_{j=1}^\infty  \frac{50^{m_{2j}}}{w(\nu_{m_{2j}})}\ge \frac{1}{50}\sum_{k\in \ZN\setminus G} \frac{50^k}{w(\nu_k)} =\infty.
	\end{equation*}
\end{proof}
The following lemma  one can find in \cite{KaSa} ch. 9, in the proof of Theorem 6.
\begin{lemma}\label{L15}
If $E_k\subset (0,1)$ are stochastically independent sets such that $|E_k|>c>0$ and the sequence $b_k>0$ satisfies $\sum_{k=1}^\infty b_k=\infty$,  then 
\begin{equation}\label{u80}
\sum_{k=1}^\infty b_k\ZI_{E_k}(x)=\infty \text{ almost everywhere.}
\end{equation}
\end{lemma}
\begin{proof}
	Let $0<c_k\le b_k$ satisfies $\sum_kc_k=\infty$ and $\sum_kc_k^2<\infty$. Observe that $\phi_k(x)=\ZI_{E_k}(x)-|E_k|$ form a stochastically independent system of orthogonal functions. It is well known that any series 
	\begin{equation*}
	\sum_{k=1}^\infty c_k\phi_k(x) \text{ with  }\sum_kc_k^2<\infty
	\end{equation*}
in such systems is almost everywhere convergent (see \cite {KaSa}, ch. 2, Theorem 9). Combining this with the relation  $\sum_{k=1}^\infty c_k|E_k|=\infty$, we get a.e. divergence of $\sum_{k=1}^\infty c_k\ZI_{E_k}(x)$ and so \e{u80}.
\end{proof}
\begin{lemma}\label{L16}
If $P\in \ZP_N$ is a polynomial of degree $N$ and $\Delta\subset \ZT$ is an interval, then
\begin{equation}
\OSC_\Delta(P)=\sup_{x,y\in \Delta}|P(x)-P(y)|\lesssim N^{3/2}|\Delta|\cdot \|P\|_2.
\end{equation}
\end{lemma}
\begin{proof}
	Suppose 
	\begin{equation*}
	P(x)=\sum_{k=1}^Na_ke^{2\pi ikx}.
	\end{equation*}
	Applying the H\"{o}lder inequality, for $x,y\in \Delta$ we get
	\begin{align*}
	|P(x)-P(y)|&\le \left(\sum_{k=1}^N|a_k|^2\right)^{1/2}\cdot \left(\sum_{k=1}^N|e^{2\pi kx}-e^{2\pi ky}|^2\right)^{1/2}\\
	&\lesssim N^{3/2}|\Delta|\cdot \|P\|_2.
	\end{align*}
\end{proof}
\begin{proof}[Proof of \cor{C2}]
	Applying \lem{L20}, we find a set of indexes $G\subset \ZN$ satisfying \e{u48} and \e{u49}.
For the sake of simplicity and without loss of generality we can suppose that $G=\ZN$. Indeed considering the general case of $G=\{r_k\}$ one just need to replace everywhere the summation $\sum_{k=1}^\infty $ by $\sum_{k\in G}$ and $\sum_{k=k_0}^\infty $ by $\sum_{k\in G\cap [k_0,\infty)}$. In some places it will also appear some indexation containing the integers $r_k$. So we suppose $G=\ZN$. Clearly there is a sequence of positive numbers $q_k\nearrow\infty$ such that
	\begin{align}
	&\frac{50^k}{q_kw(\nu_k)}\le 1,\label{u70}\\
	&\sum_{k=1}^\infty \frac{50^k}{q_kw(\nu_k)}=\infty,\quad \label{u45}\\
	&\sum_{k=1}^\infty \frac{50^k}{q_k^2w(\nu_k)}<\infty.\label{u54}
	\end{align}
Consider the intervals
	\begin{equation}
	\Delta_{k,j}=\left[\frac{j-1}{\nu_k},\frac{j}{\nu_k}\right),\quad 1\le j\le \nu_k,\quad k=1,2,\ldots.
	\end{equation}
Applying \lem{L13} with $N=(\nu_k)^2$ and $\Delta=\Delta_{k,j}$, $k\ge k_0$, we find a sequence of non-overlapping polynomials $U_{k,j,n}(x)$, $n=1,2,\ldots ,(\nu_k)^2$, such that 
\begin{align}
&\spec(U_{k,j,n})\subset (j\cdot (\nu_k)^{10},(j+1)\cdot (\nu_k)^{10}],\label{u55}\\
&\sum_{n=1}^{(\nu_k)^2}\left\|U_{k,j,n}\right\|_{L^2(\ZT)}^2\lesssim |\Delta_{k,j}|=\frac{1}{\nu_k},\label{u53}\\
&\sum_{n=1}^{(\nu_k)^2}\left|U_{k,j,n}(x)\right|\lesssim \frac{1}{(\nu_k)^2},\quad x\in \ZT\setminus \Delta_{k,j},\label{u30}\\
&\left|\left\{x\in \Delta_{k,j}:\, \max_{1\le m\le(\nu_k)^2}\Re\left(\sum_{n=1}^mU_{k,j,n}(x)\right)>\sqrt {50^k}\right\}\right|\gtrsim | \Delta_{k,j}|=\frac{1}{\nu_k}.\label{u58}
\end{align}
Observe that if
\begin{equation}\label{u35}
\Delta_{k+1,i}\cap \left\{x\in \Delta_{k,j}:\, \max_{1\le m\le(\nu_k)^2}\Re\left(\sum_{n=1}^mU_{k,j,n}(x)\right)>\sqrt {50^k}\right\}\neq\varnothing,
\end{equation}
then 
one can find an integer $m=m(k+1,j)$ with $1\le m\le(\nu_k)^{2}$ such that 
\begin{equation}\label{u34}
\Re\left(\sum_{n=1}^{m(k+1,j)}U_{k,j,n}(x)\right)>\frac{\sqrt {50^k}}{2} \text { for all }x\in \Delta_{k+1,i},
\end{equation}
since by \e{u55} any sum $\sum_{n=1}^{m}U_{k,j,n}$ is a polynomial of degree at most $(\nu_k)^{15}$ and using \lem{L16},  its oscillation on $\Delta_{k+1,i}$ can be roughly estimated by
\begin{equation}
\left\|\sum_{n=1}^{m}U_{k,j,n}\right\|_2\cdot (\nu_k)^{45/2}\cdot |\Delta_{k+1,i}|\le 1.
\end{equation}
This and \e{u35} immediately imply \e{u34}. From \e{u58} it follows that the measure of the union of all the intervals $\Delta_{k+1,i}$ satisfying \e{u35} has a lower bound $c|\Delta_{k,j}|$, where $0<c<1$ is an absolute constant.  Thus one can determine a set $E_{k,j}(\subset \Delta_{k,j})$, which is a union of some intervals $\Delta_{k+1,i}$ satisfying \e{u35} and
\begin{align}
|E_{k,j}|=d_k|\Delta_{k,j}|,\quad 0<c_1<d_k<1,\label{u36}
\end{align}
where $c_1$ is another absolute constant, while the constant $d_k$ is common for all the indexes $j=1,2,\ldots,\nu_k$.
Thus, the sets
\begin{equation*}
E_k=\bigcup_{1\le j\le \nu_k}E_{k,j},\quad k\ge k_0,
\end{equation*}
are stochastically independent and applying \lem{L15} we get 
\begin{equation}\label{u81}
\sum_{k=k_0}^\infty  \frac{50^k}{q_kw(\nu_k)}\ZI_{E_k}(x)=\infty \text { a.e.}
\end{equation} 
Using this one can choose an increasing sequence of integers $k_0<k_1<k_2<\ldots $ such that
\begin{equation}\label{u39}
\left|\left\{x\in \ZT:\, \sum_{k=k_s+1}^{k_{s+1}} \frac{50^k}{q_kw(\nu_k)}\ZI_{E_k}(x)>s\right\}\right|>1-\frac{1}{s}.
\end{equation}
Hence for almost every $x\in \ZT$ the relation 
\begin{equation}\label{u41}
\sum_{k=k_s+1}^{k_{s+1}} \frac{50^k}{q_kw(\nu_k)}\ZI_{E_k}(x)>s
\end{equation}
holds for infinitely many $s$. 
Our desired trigonometric series is
\begin{equation}\label{u27}
\sum_{k=k_0}^\infty  \frac{\sqrt{50^k}}{q_kw(\nu_k)}\sum_{j=1}^{\nu_k}\sum_{n=1}^{(\nu_k)^2}U_{k,j,n}(x),
\end{equation}
where each $U_{k,j,n}$ is considered in its trigonometric form. Note that some of the coefficients of the mentioned trigonometric series are zeros. Let us show that the coefficients of this series satisfy condition \e{a3}. Indeed, in light of \e{u48} and \e{u55} we have $w(s)\le 100w(\nu_k)$ for any $s\in \spec (U_{k,j,n})\subset (\nu_k,\nu_{k+1}]$. Thus \e{a3} may be easily deduced from \e{u53}, \e{u54} and the bound
\begin{equation*}
\sum_{k=k_0}^\infty\left(\frac{\sqrt{50^k}}{q_kw(\nu_k)}\right)^2\cdot w(\nu_k)\sum_{j=1}^{\nu_k}\sum_{n=1}^{(\nu_k)^2}\|U_{k,j,n}\|_2^2\lesssim\sum_{k=k_0}^\infty \frac{50^k}{q_k^2w(\nu_k)}<\infty.
\end{equation*}
We construct the appropriate rearrangement of series \e{u27} as follows. The collections of the trigonometric terms of our series \e{u27} involved in the groups 
\begin{equation}\label{u31}
U_{k,j,n}, \quad k_s<k\le k_{s+1},\quad 1\le j\le \nu_k,\quad 1\le n\le (\nu_k)^2
\end{equation} 
will be located in the increasing order with respect to the $s$. We just need to determine the location of each polynomial $U_{k,j,n}$ inside of the group. We do it by induction with respect to the index $k$ in \e{u31}. We leave the first group of the polynomials 
\begin{equation*}
\{U_{k_s+1,j,n}:\, 1\le j\le \nu_{k_s+1},\quad 1\le n\le (\nu_{k_s+1})^2\}
\end{equation*}
in their original order. Then suppose we have already rearranged all the polynomials $U_{k,j,n}$ corresponding to indexes $k=k_s+1,k_s+2,\ldots ,l-1$, so that the polynomials $U_{l-1,j,n}$, $n=1,2,\ldots, (\nu_{l-1})^2$ are located consecutively. Let us describe the procedure how to locate the polynomials of the next collection $\{U_{l,j,n}:\,1\le j\le \nu_{l},\, 1\le n\le (\nu_{l})^2\}$. Denote by $\Delta_{l-1,\bar j}$ the unique $(l-1)$-order interval containing the given interval $\Delta_{l,j}$ of order $l$. The following two cases are possible:

1) If $\Delta_{l,j}\subset \Delta_{l-1,\bar j}\setminus E_{l-1,\bar j}$, then we locate the polynomials $U_{l,j,n}$, $n=1,2,\ldots ,(\nu_{l})^2$, immediately after the polynomial $ U_{l-1,\bar j, (\nu_{l-1})^2}$. 

2) If $\Delta_{l,j}\subset E_{l-1,\bar j}$, then by the definition of $E_{l-1,\bar j}$ and by \e{u35}, \e{u34} for some $m=m(l,j)$ we have
\begin{equation}\label{u28}
\Re\left(\sum_{n=1}^{m(l,j)}U_{l-1,\bar j,n}(x)\right)>\sqrt{50^{l-1}}/2,\quad x\in \Delta_{l,j}.
\end{equation}
In this case we locate the polynomials $U_{l,j,n}$, $n=1,2,\ldots,(\nu_l)^2$, immediately after the $U_{l-1,\bar j,m}$. This completes the induction procedure and so the construction of the required rearrangement. It remains to prove the a.e. divergence of the series \e{u27} after the described rearrangement of the terms. For a given point $x\in \ZT$ there is a unique decreasing sequence of intervals $\Delta_{k,j_k(x)}$ containing $x$. Hence our series \e{u27} can be split into two subseries
\begin{equation}\label{u40}
\sum_{k=k_0}^\infty \frac{\sqrt{50^k}}{q_kw(\nu_k)}\sum_{n=1}^{(\nu_k)^2}U_{k,j_k(x),n}(x)+\sum_{k=k_0}^\infty \frac{\sqrt{50^k}}{q_kw(\nu_k)}\sum_{j=1}^{\nu_k}\sum_{n=1}^{(\nu_k)^2}U_{k,j,n}(x)\cdot \ZI_{\ZT\setminus\Delta_{k,j}}(x).
\end{equation}
From \e{u70} and \e{u30} it follows that
\begin{equation}\label{u47}
\sum_{k=k_0}^\infty \frac{\sqrt{50^k}}{q_kw(\nu_k)}\sum_{j=1}^{\nu_k}\sum_{n=1}^{(\nu_k)^2}|U_{k,j,n}(x)|\cdot \ZI_{\ZT\setminus\Delta_{k,j}}(x)\le \sum_{k=k_0}^\infty\frac{1}{\nu_k}<\infty.
\end{equation}
Thus we conclude that the second series in \e{u40} absolutely converges for any $x\in \ZT$. Our rearrangement of the basic series produces a rearrangement of the first subseries in \e{u40} and  it remains to prove for almost every $x\in \ZT$ such rearranged series diverges.
Denote 
\begin{align*}
A_s(x)&=\{k\in \ZN:\,  k_s<k\le k_{s+1},\,x\in E_{k,j_k(x)}\},\\
B_s(x)&=\{k\in \ZN:\, k_s<k\le k_{s+1},\, x\in \Delta_{k,j_k(x)}\setminus E_{k,j_k(x)}\}\\
&=\{k_s+1,k_s+2,\ldots,k_{s+1}\}\setminus A_s(x).
\end{align*}
According to the rearrangement construction, one can observe that there is a "restricted" partial sum (a sum of the form $\sum_{p}^q$ ) of the rearranged first  subseries of  \e{u40}, which contains all the sums of the forms
\begin{align}
&\frac{\sqrt{50^k}}{q_kw(\nu_k)}\sum_{n=1}^{m(k+1,j_{k+1}(x))}U_{k,j_k(x),n}(x),\quad k\in A_s(x),\label{u60}\\
&\frac{\sqrt{50^k}}{q_kw(\nu_k)}\sum_{n=1}^{(\nu_k)^2}U_{k,j_k(x),n}(x),\quad k\in B_s(x),\label{u38}
\end{align}
completely and there is no other terms in this partial sum. If \e{u41} holds, then according to \e{u28}, for the sum of the elements \e{u60} we obtain
\begin{align}
&\Re\left(\sum_{k\in A_s(x)}\frac{\sqrt{50^k}}{q_kw(\nu_k)}\sum_{n=1}^{m(k+1,j_{k+1}(x))}U_{k,j_k(x),n}(x)\right)\\
&\qquad\qquad\ge \frac{1}{2}\sum_{k=k_s+1}^{k_{s+1}}\frac{50^k}{q_kw(\nu_k)}\ZI_{E_k}(x)>\frac{s}{2}.\label{u59}
\end{align}
As for the elements \e{u38}, they form an a.e. absolutely convergence series. Indeed,  we have a pointwise bound
\begin{align}
\sum_{k=k_0}^\infty \frac{\sqrt{50^k}}{q_kw(\nu_k)}\left|\sum_{n=1}^{(\nu_k)^2}U_{k,j_k(x),n}(x)\right|&=\sum_{k=k_0}^\infty \frac{\sqrt{50^k}}{q_kw(\nu_k)}\left|\sum_{n=1}^{(\nu_k)^2}U_{k,j_k(x),n}(x)\right|\cdot \ZI_{\Delta_{k,j_k(x)}}(x)\\
&\le \sum_{k=k_0}^\infty \frac{\sqrt{50^k}}{q_kw(\nu_k)}\sum_{j=1}^{\nu_k}\left|\sum_{n=1}^{(\nu_k)^2}U_{k,j,n}(x)\right|\cdot \ZI_{\Delta_{k,j}}(x)\\
&= \sum_{k=k_0}^\infty R_k(x),\label{u71}
\end{align}
and then using \e{u70}, \e{u53} and an orthogonality argument, we obtain
\begin{align*}
\sum_{k=k_0}^\infty \|R_k\|_2&=\sum_{k=k_0}^\infty \frac{\sqrt{50^k}}{q_kw(\nu_k)}\left\|\sum_{j=1}^{\nu_k}\left|\sum_{n=1}^{(\nu_k)^2}U_{k,j,n}(x)\right|\cdot \ZI_{\Delta_{k,j}}(x)\right\|_2\\
&=\sum_{k=k_0}^\infty \frac{\sqrt{50^k}}{q_kw(\nu_k)}\left(\sum_{j=1}^{\nu_k}\left\|\sum_{n=1}^{(\nu_k)^2}U_{k,j,n}(x)\cdot \ZI_{\Delta_{k,j}}(x)\right\|_2^2\right)^{1/2}\\
&\le \sum_{k=k_0}^\infty \frac{\sqrt{50^k}}{q_kw(\nu_k)}\left(\sum_{j=1}^{\nu_k}\left\|\sum_{n=1}^{(\nu_k)^2}U_{k,j,n}\right\|_2^2\right)^{1/2}\\
&\lesssim \sum_{k=k_0}^\infty \frac{\sqrt{50^k}}{q_kw(\nu_k)}=\sum_{k=k_0}^\infty \frac{1}{\sqrt{50^k}}\cdot \frac{50^k}{q_kw(\nu_k)}\le \sum_{k=k_0}^\infty \frac{1}{\sqrt{50^k}}<\infty
\end{align*}
that implies the a.e. convergence of series \e{u71}. Combining \e{u59} with the a.e. absolutely convergence of the series consisting of the terms \e{u38}, we conclude that the first subseries in \e{u40} diverges for a.e. $x\in \ZT$. This completes the proof of \cor{C2}.
\end{proof}

\bibliographystyle{plain}


\end{document}